\def\mystyle{}

\if\mystyle l
\documentclass[reqno,landscape]{amsart}
\else
\documentclass[reqno]{amsart}
\fi

\usepackage{begnac}

\renewcommand{\det}{\operatorname{det}}

\begin{document}

\title{Estimates on volumes of homogeneous polynomial spaces}

\author{Itaï \textsc{Ben Yaacov}}

\address{Itaï \textsc{Ben Yaacov} \\
  Université de Lyon \\
  Université Claude Bernard -- Lyon 1 \\
  Institut Camille Jordan, CNRS UMR 5208 \\
  43 boulevard du 11 novembre 1918 \\
  69622 Villeurbanne Cedex \\
  France}

\urladdr{\url{http://math.univ-lyon1.fr/~begnac/}}

\thanks{Author supported by ANR project ValCoMo (ANR-13-BS01-0006) and ERC Grant no.\ 291111.}

\svnInfo $Id: Volume.tex 3711 2018-01-25 14:05:25Z begnac $
\thanks{\textit{Revision} {\svnInfoRevision} \textit{of} \today}

%\date{\today}

\begin{abstract}
  In this paper we develop the ``local part'' of our local/global approach to globally valued fields (GVFs).
  The ``global part'', which relies on these results, is developed in a subsequent paper.

  We study \emph{virtual divisors} on projective varieties defined over a valued field $K$, as well as \emph{sub-valuations} on polynomial rings over $K$ (analogous to homogeneous polynomial ideals).
  We prove a Nullstellensatz-style duality between projective varieties equipped with virtual divisors (analogous to projective varieties over a plain field) and certain sub-valuations on polynomial rings over $K$ (analogous to homogeneous polynomial ideals).
  Our main result compares the \emph{volume} of a virtual divisor on a variety $W$, namely its $(\dim W + 1)$-fold self-intersection, with the asymptotic behaviour of the volume of the dual sub-valuation, restricted to the space of polynomial functions of degree $m$, as $m \rightarrow \infty$.

  \textbf{This is work in progress}.
\end{abstract}

\maketitle

\tableofcontents

\section*{Introduction}

\section{Definability and stability in algebraically closed (valued) fields}
\label{sec:Definability}

Let us recall a few facts from model theory, as it pertains to algebraically closed fields (this is all essentially folklore).
By a \emph{formula} $\varphi(X)$, in a tuple of indeterminates $X$, we mean a Boolean combination of polynomial equalities $f(X) = 0$: if $a$ is a tuple in a field $K$ of the appropriate length, then $\varphi(a)$ is either True or False.
We often split the indeterminates in several families, with the notation $\varphi(X,Y)$.
We may substitute elements of $K$ for some of the variables, obtaining a \emph{formula with parameters} $\varphi(X,b)$.
When $A \subseteq K$, we write $\varphi(A,b)$ for $\bigl\{ a \in A : \varphi(a,b) \bigr\}$ (here and later, $a \in A$ should be understood loosely as ``$a$ is a tuple in $A$ of the appropriate length'').
When $A = K$ and $b \in K$, the set $\varphi(K,b)$ is called a \emph{definable set} in $K$, and is just a constructible set defined over $K$.
When no ambiguity may arise, we sometimes omit the parameters $b$ from the notation, saying that $\varphi(X)$ is a formula \emph{over} $K$.

Algebraically closed fields have \emph{quantifier elimination}: if $\varphi(X,Y)$ is a formula in $X,Y$, then $\exists Y \, \varphi(X,Y)$ is (equivalent to) a formula $\psi(X)$ (once this holds for formulae without parameters, it also holds for ones with parameters).
Indeed, this is just Chevallay's Theorem: a coordinate projection of a constructible set in an algebraically closed field is again a constructible set.
Since formulae are closed under negation, the same is true of $\forall Y \, \varphi(X,Y)$, and of any other expression constructed using quantifiers or Boolean operations.
(In fact, what we defined as a formula is what is usually called a \emph{quantifier-free} formula, but by quantifier elimination the two notions agree.)
Notice that if $D$ is a definable set, defined by $\psi(Y)$, and $\varphi(X,Y)$ is another formula, then $(\exists Y \in D) \ \varphi(X,Y)$ is again a formula (since it is the same as $\exists Y \ \bigl( \psi(Y) \wedge \varphi(X,Y) \bigr)$), and similarly for $(\forall Y \in D) \ \varphi(X,Y)$.
One important consequence of quantifier elimination is the following: if $\varphi(X)$ and $\psi(X)$ are two formulae over $K$, such that $\varphi(K) = \psi(K)$, then $\varphi(L) = \psi(L)$, for any larger field $L$.
Indeed, we may assume that $L$ is algebraically closed, and the formula the formula $\forall X \ \varphi(X) \leftrightarrow \psi(X)$ is true in $K$, so also in $L$.

The class of algebraically closed fields is \emph{stable}.
This has many equivalent characterisations, one of which is the following: if $L/K$ is an extension of algebraically closed fields, and $\varphi(X,b)$ is a formula with parameters $b \in L$, then the set $\varphi(K,b)$ is definable in $K$ (this is far from being true for any kind of structure: for example, $\bQ \subseteq \bR$ is an extension of dense linear orders without endpoints, which also have quantifier elimination, but the set $\{q \in \bQ : q < \pi\}$ is not definable in $\bQ$).
An equivalent characterisation of stability is via the existence of a (necessarily unique) \emph{notion of independence} of two structures $L$ and $M$ over a common substructure $K$, denoted $L \ind_K M$, satisfying certain axioms that we do not state here (see for example Pillay \cite{Pillay:GeometricStability}).
Thus, for example, properties of linear independence over a common subspace imply that the class of vector spaces over some fixed field is stable, and stochastic independence yields stability for probability algebras.
In the class of algebraically closed fields, when $L$ and $M$ are subfields of some large ambient field, and $K$ is a common algebraically closed sub-field, we say that $L \ind_K M$ if $L$ and $M$ are linearly disjoint over $K$, i.e., if they generate the algebra $L \otimes_K M$ inside the ambient field (since $K$ is algebraically closed, $L \otimes_K M$ is always an integral domain).
It is a general fact that in a stable class of structures, the following are equivalent:
\begin{enumerate}
\item We have $L \ind_K M$.
\item If $\varphi(X,b)$ is a formula with parameters $b \in L$, then $\varphi(M,b)$ can be defined by a formula $\psi(X,c)$ with a parameter $c \in K$ (so $\psi(X,c)$ also defines, in $K$, the restriction $\varphi(K,b)$).
\end{enumerate}

Let us see how this is proved for fields.
It is enough to consider the special case of a polynomial equality $f(X,b) = 0$, where $f(X,Y) \in \bZ[X,Y]$ and $b \in L$.
We may assume that $f$ is homogeneous in $X$, and applying the Veronese map of degree $d$, we may assume that $f$ is linear in $X$.
But then the set $f(K,b) = 0$ is a linear subspace of $K^m$, and it is easily definable, say, as the zero set of a family of linear forms with coefficients in $K$.
If $M$ is another extension of $K$, linearly disjoint of $L$ over $K$, then the linear form $f(X,b)$ vanishes on $c \in M^m$ if and only if $c$ can be expressed as an $M$-linear combination of tuples in $f(K,b) = 0$.
It follows that the same intersection of zero sets of linear forms over $K$ also defines $f(M,b) = 0$.
Going back to our original setting where $f$ is an arbitrary polynomial, if $W$ is an algebraic set defined over $K$ (or indeed, any set definable with parameters in $K$), and $L \ind_K M$, then $f(X,b) = 0$ has a solution in $W(K)$ if and only if it has a solution in $W(M)$ (by quantifier elimination).

\medskip

We are going to be interested in a similar situation, but in the context of algebraically closed fields equipped with a non-trivial valuation $v\colon K \rightarrow \bR \cup \{\infty\}$ (equivalently, with a non-trivial, non-Archimedean absolute value).
Therefore, from this point onward, a \emph{valuation} will always be in the ordered group $(\bR,+,<)$ (not necessarily onto).
Fields equipped with such a valuation are studied, from a model-theoretic point of view, in \cite{BenYaacov:MetricValuedFields}.
Following the notation there, we shall write $K \vDash MVF$ to say that $K$ is a complete valued field, and $K \vDash ACMVF$ to say that $K$ is, in addition, algebraically closed with a non-trivial valuation.

The presence of $\bR$ as a fixed object requires us to replace classical Boolean logic with real-valued \emph{continuous} logic (see \cite{BenYaacov-Usvyatsov:CFO,BenYaacov-Berenstein-Henson-Usvyatsov:NewtonMS}).
Thus, the set of possible truth values for a formula $\varphi(X)$ is no longer $\{T,F\}$, but some compact interval of $\bR$, and formulae are closed under continuous, rather than Boolean, combinations.
For this closure property we allow infinite continuous combinations, or, equivalently, finite continuous combinations and uniform limits.
We may also allow $[-\infty,\infty]$ as a truth value space, with the notion of uniform convergence corresponding to its compact topology (equivalently, arising from any homeomorphism with a compact interval of $\bR$).

\begin{conv}
  \label{conv:Predicate}
  Breaking with usual terminology, by a \emph{predicate} we mean a map from a set (possibly a Cartesian product of sets) to $[-\infty,\infty]$.
  A special case of this is the more familiar notion of a \emph{Boolean predicate}, which is a map into $\{0,1\}$ (or $\{\textit{True},\textit{False}\}$, where \textit{True} is identified with $0$ and \textit{False} with $1$).
\end{conv}

One last complication: the structure cannot be taken to be the field $K$ itself, since it is unbounded, nor the valuation ring, since its ultra-powers are not necessary valuation rings.
We take instead the disjoint union of $\bP^n(K)$ for all $n$, which will be denoted $\bP(K)$ (each $\bP^n$ is a \emph{sort} of $\bP(K)$).
Let us now define what (quantifier-free) formulae are in this structure.

\begin{ntn}
  \label{ntn:ValuationTuple}
  For $a \in K^m$ we let
  \begin{gather*}
    \tilde{v}(a) = \min_{i < m} v(a_i).
  \end{gather*}
  We shall use this mainly in two contexts, for affine points, and for polynomials, identified with their tuple of coefficients.
\end{ntn}

\begin{dfn}
  \label{dfn:AtomicFormula}
  Let $f \in \bZ[X_0,X_1,\ldots]$, where each of $X_i$ is a tuple of $n_i+1$ indeterminates, such that $f$ is homogeneous of degree $d_i$ in each $X_i$.
  Then, for $\xi_i = [x_i] \in \bP^{n_i}(K)$ we define
  \begin{gather*}
    vf(\xi_0,\xi_1,\ldots) = v \circ f(x_0,x_1,\ldots) - \sum d_i \tilde{v}(x_i) \in [0,\infty],
  \end{gather*}
  noting that it depends only on the points $\xi_0,\xi_1,\ldots$ (and not on the representatives).
\end{dfn}

Each $vf$ is considered an \emph{atomic formula}, and any continuous combinations thereof is a \emph{formula}.
Thus, each formula defines a predicate on $\bP(K)$ (i.e., on a Cartesian product of sorts of $\bP(K)$).
Notice that a polynomial $f \in K[X_0,\ldots]$ can be written as $g(b,X_0,\ldots)$, where $g \in \bZ[Y,X_0,\ldots]$ is linear in $Y$ and $b$ are the coefficients.
In this case
\begin{gather*}
  vf(\xi_0,\ldots) = vg\bigl( [b], \xi_0, \ldots \bigr) + \tilde{v}(b),
\end{gather*}
so allowing coefficients in $K$ does not change the expressive power.

\begin{fct}[{\cite[Theorem~2.4 and Proposition~2.16]{BenYaacov:MetricValuedFields}}]
  \label{fct:QuantifierEliminationACMVF}
  Algebraically closed metric valued fields admit quantifier elimination.
  In other words, if $\varphi(\xi,\zeta)$ is a formula (where each of $\xi$ and $\zeta$ represents a tuple of indeterminate projective points), then $\inf_\zeta \ \varphi(\xi,\zeta)$ is (equivalent to) a formula in every $K \vDash ACMVF$.
  Consequently, $\sup_\zeta \ \varphi(\xi,\zeta)$, as well as any more complex expression constructed using these \emph{quantifiers} and continuous combinations, is a formula.

  Moreover, every algebraic $W \subseteq \bP^n$ is a \emph{definable set} (in the sense of continuous logic), that is to say that $\inf_{\zeta \in W}\varphi(\xi,\zeta)$ is also a formula over $K$.
\end{fct}

Notice that \cite{BenYaacov:MetricValuedFields} uses multiplicative notation, whereas here we use additive notation, but since $e^{-t} \colon [0,\infty] \rightarrow [0,1]$ is a monotone homeomorphism, this changes nothing.
Given an incomplete valued field $K$, it has a unique completion $\widehat{K}$, and the quantifiers $\inf$ and $\sup$ evaluate the same in $K$ and in $\widehat{K}$.
We may therefore assume that all the fields in question are complete (and if an incomplete one arises, replace it tacitly with its completion).

Let us show that the class $ACMVF$ is stable using the characterisation given above (a different argument, by counting types, is given in \cite{BenYaacov:MetricValuedFields}).
In other words, if $L/K$ is an extension of such fields, and $\varphi(X,b)$ is a formula with parameters in $L$, then its restriction to $K$ is definable, in $K$, by a formula with parameters there.
The fundamental tool for doing this is the notion of a valued vector space.

\begin{dfn}
  \label{dfn:VectorSpaceValuation}
  Let $K \vDash MVF$, $E$ a vector space over $K$.
  A \emph{valuation} on $E$ is a function $u\colon E \rightarrow \bR \cup \{\infty\}$ satisfying, for all $x,y \in E$ and $a \in K$:
  \begin{enumerate}
  \item $u(a x) = v(a) + u(x)$
  \item $u(x+y) \geq u(x) \wedge u(y)$
  \end{enumerate}
  Its \emph{kernel} is the subspace
  \begin{gather*}
    \ker u = \{ x \in E : u(x) = \infty\}.
  \end{gather*}
  If $\ker u = \{0\}$, then we say that $u$ is \emph{reduced}.
\end{dfn}

\begin{rmk}
  \label{rmk:VectorSpaceValuation}
  One might think that what we call a valuation should be called a \emph{semi-valuation}, reserving the term \emph{valuation} to reduced ones.
  There are, however, several good reasons for our choice terminology, such as the analogy between ideals and (sub-)valuations on rings, which is explored below.
\end{rmk}

\begin{dfn}
  \label{dfb:VectorSpaceValuationQuotient}
  Let $E$ be a valued vector space and $F \subseteq E$ a subspace.
  We define a \emph{quotient valuation} on either $E$ or $E/F$ (both points of view may be useful) by:
  \begin{gather*}
    (u/F)(x) = u(x + F) = \sup \, \bigl\{ u(y) : y \in x + F \bigr\}.
  \end{gather*}
\end{dfn}

It is easy to check that the quotient valuation is indeed a valuation.
Clearly, $\tilde{v}$ is a reduced valuation on $K^m$.
Conversely:

\begin{lem}
  \label{lem:VectorSpaceValuationBasis}
  Any reduced valuation $u$ on $K^m$ is arbitrarily close, up to a change of coordinates by a triangular matrix, to $\tilde{v}$.

  If $u$ is a reduced valuation on $E$ and $F \subseteq E$ is finite-dimensional, then the quotient valuation on $E/F$ is reduced as well.
\end{lem}
\begin{proof}
  We prove the first item by induction on $m$, with $m = 0$ being trivial.
  For $m+1$, let $(e_i : i \leq m)$ denote a basis of $E$, and $F = \Span(e_i : i < m)$.
  Let $\varepsilon > 0$.
  Applying a change of coordinates to $F$ by a triangular matrix, we may assume that $|u(y) - \tilde{v}(y)| < \varepsilon$ for all $y \in F$.
  Let $\alpha = u(e_m + F)$.
  If $\alpha = \infty$, then there exist sequences $(a_{i,n})$ for $i < m$ such that $u\left( e_m + \sum_{i<m} a_{i,n} e_i \right) \rightarrow \infty$ as $n \rightarrow \infty$.
  By our assumption regarding $u$ on $F$ each of the sequences $(a_{i,n})$ must be Cauchy, and therefore converge to some $a_i$.
  But then $u\left( e_m + \sum_{i<m} a_i e_i \right) = \infty$, a contradiction.

  Therefore $\alpha < \infty$, and we may choose $x \in e_m + F$ such that $\alpha < u(x) + \varepsilon$.
  Taking $b \in K$ such that $\alpha - \varepsilon < -v(b) < u(x)$, we have $0 < u(bx) \leq \alpha + v(b) < \varepsilon$.
  Replacing $e_m$ with $bx$, we complete a triangular change of basis for $K^{m+1}$, and now $0 < u(e_m) < \alpha < \varepsilon$.
  Let $y \in E$.
  If $y \in F$, then we already have $|u(y) - \tilde{v}(y)| < \varepsilon$.
  Otherwise, we may assume that $y \in e_m + F$.
  If $u(y-e_m) \leq 0$, then $u(y) = u(y-e_m)$, and in any case $u(y) \leq \varepsilon$.
  Therefore
  \begin{gather*}
    u(y)
    \leq 0 \wedge u(y-e_m) + \varepsilon
    \leq \tilde{v}(e_m) \wedge \bigl( \tilde{v}(y-e_m) + \varepsilon \bigr) + \varepsilon
    \leq \tilde{v}(y) + 2\varepsilon.
  \end{gather*}
  Since clearly $u \geq \tilde{v}$, we reach the desired conclusion.

  The second assertion follows from our argument that $\alpha < \infty$.
\end{proof}

\begin{dfn}
  \label{dfn:ValuationGeneric}
  Let $L/K$ be an extension.
  \begin{enumerate}
  \item Say that a tuple $a \in L^m$ is \emph{weakly $v$-generic} over $K$, if for every polynomial $h \in K[X]_d$:
    \begin{gather*}
      v\bigl( h(a) \bigr) \leq \tilde{v}(h).
    \end{gather*}
    Notice that this implies that $\tilde{v}(a) \leq 0$.
  \item Say that $a$ is \emph{$v$-generic} over $K$ if for every polynomial $h \in K[X]_d$:
    \begin{gather*}
      v\bigl( h(a) \bigr) = \tilde{v}(h).
    \end{gather*}
    Equivalently, if it is weakly $v$-generic and $\tilde{v}(a) = 0$.
  \end{enumerate}
\end{dfn}

\begin{lem}
  \label{lem:ValuationGeneric}
  Let $K$ be a valued field.
  \begin{enumerate}
  \item One can always adjoin to $K$ new $v$-generic elements.
  \item Assume that $a \in L^m$ is $v$-generic over $K$, and $b \in K^m$.
    Then $a + b \in L^m$ is weakly $v$-generic over $K$.
  \end{enumerate}
\end{lem}
\begin{proof}
  It is a standard (and easy) fact that $\tilde{v}$ is multiplicative on $K[X]$, namely, that $\tilde{v}(fg) = \tilde{v}(f) + \tilde{v}(g)$ for all $f,g \in K[X]$.
  This gives rise to a valuation $w(f/g) = \tilde{v}(f) - \tilde{v}(g)$ on the fraction field $L = K(X)$, for which $X$ are $v$-generic over $K$.
  The rest is immediate.
\end{proof}

\begin{prp}
  \label{prp:PolynomialTrace}
  Let us fix a degree $d \geq 1$, and let $W \subseteq \bP^n$ be an algebraic set defined over $K$.
  Then the following are equivalent for a function $\eta\colon W(K) \rightarrow \bR \cup \{\infty\}$:
  \begin{enumerate}
  \item
    \label{item:PolynomialTraceIn}
    The function $\eta$ is a uniform limit on $W(K)$ of functions of the form $\min_{i<N} \ vg_i\rest_{W(K)}$, where $g_i \in K[X]_d$.
  \item
    \label{item:PolynomialTraceOut}
    There exists an extension $L/K$ and polynomial $f \in L[X]_d$ such that $\eta = vf\rest_{W(K)}$.
  \item
    \label{item:PolynomialTraceLinearValuation}
    [When $d = 1$] There exists a valuation $u$ on $K^{n+1}$ such that $\eta(\xi) = u(x) - \tilde{v}(x)$ for all $\xi = [x] \in W(K)$.
  \end{enumerate}
  Moreover, if $\eta$ is finite then it is bounded in $\bR$, and in the second item one may always take $N = \binom{n+d}{d}$.
\end{prp}
\begin{proof}
  Applying a Veronese map, and replacing $n$ with $\binom{n+d}{d}-1$, we may assume that $d = 1$.
  We may further assume that $W(K)$ does not lie in any hyperplane of $\bP^n$.
  \begin{cycprf*}[1]
  \item
    Let us first assume that $\eta = \min_{i<N} \ vg_i\rest_{W(K)}$ for some $g_i \in K[X]_1$.
    Let $L/K$ be an extension containing a tuple $a$ of $N$ $v$-generic elements, and let $f = \sum a_i g_i \in L[X]_1$.
    Then $\eta = vf\rest_{W(K)}$.
    Also, a direct calculation shows that $\tilde{v}(f) = \min \ \tilde{v}(g_i) \leq \inf \eta$.

    Let us show that the difference $\inf \eta - \tilde{v}(f)$ is bounded by a constant $M$ which depends only on $W$, and not on $\eta$ or the $g_i$.
    Indeed, assume not.
    Then there exists a sequence of $(g_{m,i} : m \in \bN, \ i < N_m) \subseteq K[X]_1$ such that $\min_i \ \tilde{v}(g_{m,i}) = 0$ for all $m$, and yet $vg_{m,i}\rest_{W(K)} \geq m$ for all $m,i$.
    In particular, there exists a sequence $(g_m) \subseteq K[X]_1$ such that $\tilde{v}(g_m) = 0$ and $\inf_{\xi \in W} \, vg_i(\xi) \geq m$ holds in $K$.
    Let $M = K^\cU \supseteq K$ be an ultra-power, and let $g \in M[X]_1$ be the limit of the sequence $(g_m)$.
    Then $\tilde{v}(g) = 0$ and $\inf_{\xi \in W} \, vg(\xi) \geq m$ for all $m$, i.e., $g$ vanishes on $W$.
    Forgetting the valuation, $M/K$ is an extension of algebraically closed fields.
    By quantifier elimination for those, there exists $h \in K[X]_1$ which vanishes on $W$, contradicting our hypothesis.

    We can now prove the converse in the general case.
    Indeed, we assume that $\eta$ is a uniform limit on $W(K)$ of functions $\eta_k = \min_{i < N_k} \ vg_{k,i}\rest_{W(K)}$.
    Say, in particular, that $|\eta - \eta_k| \leq 1$ for all $k$.
    Express each $\eta_k$ as $vf_k\rest_{W(K)}$ where $f_k \in L_k[X]_1$ and $\tilde{v}(f_k) \geq \inf \eta - M - 1$.
    Let $L = \prod L_k/\cU$ be an ultra-product, and $f \in L[X]_d$ the image of the sequence $(f_k)$ (for this to exist we require the common lower bound for $\tilde{v}(f_k)$).
    It follows that $\eta = vf\rest_{W(K)}$.
  \item
    If $f$ is linear, then $u(x) = v \circ f(x)$ is a valuation on $K^{n+1}$.
  \item[\impfirst]
    Dividing by $\ker u$ we may assume that $u$ is reduced.
    Then, up to a change of coordinates, it is arbitrarily close to $\tilde{v}$.
    In other words, $u(x)$ is arbitrarily close to $\min \ v \circ g_i(x)$ for a family $(g_i : i < n) \subseteq K[X]_1$.
    It follows that $\min_{i<n} \ vg_i\rest_{W(K)}$ is as close as desired to $\eta$.
  \end{cycprf*}
  Finally, let us prove that if $\eta$ is finite, then it is bounded.
  For this we may assume that $\eta = \min_{i<N} \ vg_i\rest_{W(K)}$ as above.
  Indeed, if $\sup \eta = \infty$, then there exist $\xi_m \in W(K)$ such that $vg_i(\xi_m) \geq m$ for all $i$.
  In an ultra-power $M$ we find $\xi \in W(M)$ on which all the $g_i$ vanish.
  Forgetting the valuation as above, we find $\xi \in W(K)$ on which all the $g_i$ vanish, so $\eta(\xi) = \infty$, a contradiction.
\end{proof}

In particular, for every homogeneous polynomial $f(X) \in L[X]$, the restriction of $vf$ to $\bP^n(K)$ is uniformly approximated by formulae with parameters in $K$, i.e., $vf\rest_{\bP^n(K)}$ is definable in $K$.
The Segre embedding allows us to replace a homogeneous polynomial in several families of indeterminates with one in a single family, and we conclude that the restriction to $K$ of any atomic formula over $L$ is definable in $K$, and the same follows for every formula.
It follows that the theory $ACMVF$ is stable.

Consider now $L,M \vDash MVF$, both embedded in a large valued field, along with a common sub-field $K \vDash ACMVF$.
We want to characterise when $L \ind_K M$, i.e., when the restriction of a formula over $L$ to $M$ is definable over $K$.
As above, it suffices to consider the case of a formula $vf$ where $f \in L[X]$ is homogeneous in a single family of indeterminates.
For this, we require one last tool.

\begin{dfn}
  \label{dfn:VectorSpaceValuationTensor}
  If $E$ and $F$ are two valued vector spaces then the \emph{tensor product valuation} on $E \otimes_K F$ as the least one satisfying $u(x \otimes y) \geq u_E(x) + u_F(y)$ for every simple tensor $x \otimes y$.
  It is sometimes denoted $u_E \otimes u_F$.
\end{dfn}

\begin{lem}
  \label{lem:VectorSpaceValuationTensor}
  With the hypotheses of \autoref{dfn:VectorSpaceValuationTensor}:
  \begin{enumerate}
  \item
    The tensor product valuation exists.
  \item
    We have $(u_E \otimes u_F)(x \otimes y) = u(x) + u(y)$ for all simple tensors, and more generally,
    \begin{gather*}
      (u_E \otimes u_F)(z) = \sup_{} \Bigl\{ \min \bigl( u_E(x_i) + u_F(y_i) \bigr) : z = \sum x_i \otimes y_i \Bigr\},
    \end{gather*}
    where the supremum is taken over all presentations of $z$ as a sum of simple tensors.
  \item
    \label{lem:VectorSpaceValuationTensorOneHat}
    When $F = K^m$ is equipped with $\tilde{v}$ then $E \otimes_K K^m = E^m$ and $u_E \otimes \tilde{v} = \tilde{u}_E$, i.e.:
    \begin{gather*}
      (u_E \otimes \tilde{v})(x_0,\ldots,x_{m-1}) = \min_i u_E(x_i).
    \end{gather*}
    (This is the \emph{direct sum valuation} on $E \oplus \cdots \oplus E$.)
  \item If $E_0 \subseteq E$ and $F_0 \subseteq F$ are sub-spaces, then
    \begin{gather*}
      (u_E\rest_{E_0}) \otimes (u_F\rest_{F_0}) = (u_E \otimes u_F)\rest_{E_0 \otimes_K F_0}.
    \end{gather*}
  \item We have $\ker (u_E \otimes u_F) = (\ker u_E) \otimes_K (\ker u_F)$.
  \end{enumerate}
\end{lem}
\begin{proof}
  Let us consider first the special case where $E = K^m$ and $F = K^n$ are equipped with $\tilde{v}$, and $E_0 = K^{m_0} \times \{0\}$, $F_0 = K^{n_0} \times \{0\}$.
  Identifying $E \otimes_K F$ with $K^{mn}$ we have $\tilde{v}(x \otimes y) = \tilde{v}(x) + \tilde{v}(y)$, and $\tilde{v}$ on $K^{mn}$ is least satisfying $\tilde{v}(x_i \otimes y_j) \geq 0$ for $x_i$ and $y_j$ in the respective standard bases, so $\tilde{v} = \tilde{v} \otimes \tilde{v}$ (on the appropriate spaces).
  In this case, all our assertions are easy to check.
  The case of any two reduced valuations follows by \autoref{lem:VectorSpaceValuationBasis} (since the change of basis is by a triangular matrix, we reduce to the special forms of $E_0$ and $F_0$ assumed earlier).
  In the general case, we divide by the kernels.
\end{proof}

If $L$ and $M$ are field extensions of $K$, then $L \otimes_K M$ is a ring, and it is immediate to check that $u(cd) \geq u(c) + u(d)$ for all $c,d \in L \otimes_K M$.
The main theorem of \cite{BenYaacov:ValuedFieldTensorProduct} asserts that, when $K \vDash ACMVF$, we have $u(cd) = u(c) + u(d)$, giving rise to a natural valuation $v(c/d) = u(c) - u(d)$ on $\Frac(L \otimes_K M)$.
(It is proved using quantifier elimination for $ACVF$, in Boolean logic, but can also be proved using quantifier elimination in $ACMVF$, as well as via other methods, as in Poineau \cite{Poineau:Angelique}.)

\begin{prp}
  \label{prp:ValuedFieldIndependence}
  Let $K,L,M \vDash ACMVF$, where both $L$ and $M$ are embedded in some large valued field and $K$ is a common sub-field.
  Then the following are equivalent:
  \begin{enumerate}
  \item The compositum $LM$ is $\Frac(L \otimes_K M)$ (as a valued field).
  \item For every homogeneous polynomial $f \in M[X]$, the restriction of $vf$ to $L$ is definable with parameters in $K$.
  \end{enumerate}
\end{prp}
\begin{proof}
  In one direction, we may assume, as in the proof of \autoref{prp:PolynomialTrace}, that $f(X)$ is linear.
  We let $u_1(x) = v \circ f(x)$ for $x = K^N$, and similarly $u_2(x)$ for $x \in L^N$.

  Let $x \in L^N$ be given.
  Let $E \subseteq L$ be the $K$-vector space generated by the coefficients of $x$, and let $(y_i : i < \ell)$ be a basis for $E$.
  Assume first that $(y_i)$ induces an isomorphism between $E$ and $(K^\ell,\tilde{v})$, and express $x = \sum y_i x_i$, where $x_i \in K^N$.
  Then $f(x) = \sum y_i f(x_i) \in L \otimes_K M$, and by \autoref{lem:VectorSpaceValuationTensor}\autoref{lem:VectorSpaceValuationTensorOneHat} we have
  \begin{gather*}
    (v_L \otimes u_1)(x)
    = \min \, u_1(x_i)
    = \min \, v_M\bigl( f(x_i) \bigr)
    = (v_L \otimes v_M)\bigl( f(x) \bigr)
    = u_2(x).
  \end{gather*}
  We reduce the general case to this special one (approximately) by \autoref{lem:VectorSpaceValuationBasis}, so $u_2 = v_L \otimes u_1$.

  In particular, we may now reduce to the case where $u_1$ is a valuation.
  We then have a family of linear functions $g_i$ over $K$ such that $u_1(x)$ is arbitrarily close to $\min v \circ g_i(x) = \tilde{v}_K \bigl( g_i(x) : i < N \bigr)$ for $x \in K^N$.
  Then, for $x \in L^N$, $u_2(x)$ is as close to $\tilde{v}_L \bigl( g_i(x) : i < N \bigr)$, i.e., to $\min v \circ g_i(x)$ again.
  Thus $vf$ restricted to $L$ is definable with parameters in $K$.

  For the converse, let us compare the compositum $LM$ (in the ambient valued field) with $\Frac(L \otimes_K M)$.
  Consider a homogeneous polynomial $f \in M[X]$.
  The restriction of $vf$ to $\bP^n(K)$ is definable, in $K$ (with parameters there), say by $\varphi(\xi)$.
  In both $LM$ and $\Frac(L \otimes_K M)$, the restriction of $vf$ to $\bP^n(L)$ is definable by formulae with parameters in $K$, say $\psi_1(\xi)$ and $\psi_2(\xi)$.
  But then $\sup_\xi \bigl| \varphi(\xi) - \psi_i(\xi) \bigr|$ is a formula with parameters in $K$, evaluating to zero.
  Therefore it must also evaluate to zero in $L$ (we are making a non-trivial use of quantifier elimination here), so $\psi_1(\xi) = \psi_2(\xi)$ for all $\xi \in \bP^n(L)$.
  In other words, for $\xi \in \bP^n(L)$, $vf(\xi)$ is the same in $LM$ and in $\Frac(L \otimes_K M)$.
  It follows that the valuation on $LM$ is the one induced from $\Frac(L \otimes_K M)$, as stated.
\end{proof}

As in the case of pure fields, we can reduce from arbitrary formulae to ones of the form $vf(\xi)$.
We concede that $LM \cong \Frac(L \otimes_K M)$ as valued fields if and only if the restriction of every formula over $L$ to $M$ is definable over $K$, i.e., if and only if $L \ind_K M$ in the sense of model-theoretic stability.
Let us conclude with a technical result which will be used later on.

\begin{lem}
  \label{lem:ValuationGenericIndependent}
  Assume that $L \ind_K M$, where $K, L, M \vDash ACMVF$, and let $a \in L^m$ be (weakly) $v$-generic over $K$.
  Then it is also (weakly) $v$-generic over $M$.
\end{lem}
\begin{proof}
  This can be calculated directly, but let us give an argument using stability.
  So fix a degree $d$ and let $b \in L^N$ consist of all monomials of degree $d$ in $a$.
  Let $f(X) = \sum b_i X_i \in L[X]_1$, and let $\varphi(\xi)$ be the formula, with parameters in $K$, defining the restriction of $vf$ to $\bP(K)$.
  Since $L \ind_K M$, it also defines the restriction of $vf$ to $\bP(M)$.
  By hypothesis we have $\sup_\xi \varphi(\xi) \leq 0$ in $K$, so also in $M$.
  Doing this for all $d$, we see that $a$ is weakly $v$-generic over $M$.
  It is then $v$-generic over either field if and only if, in addition, $\tilde{v}(a) = 0$.
\end{proof}

\section{Virtual divisors and virtual chains}
\label{sec:VirtualDivisors}

Throughout, $K \vDash ACMVF$.

\begin{ntn}
  \label{ntn:Hat}
  If $f \in K[X]_m$, with $m \geq 1$, then for $\xi = [x] \in \bP^n(K)$ we write
  \begin{gather*}
    \hat{f}(\xi) = \frac{vf(\xi)}{m} = \frac{v \circ f(x)}{m} - \tilde{v}(x).
  \end{gather*}
\end{ntn}

\begin{dfn}
  \label{dfn:VirtualDivisor}
  Let $W(K) \subseteq \bP^n(K)$ be an algebraic set defined over $K$.
  A \emph{virtual divisor} of degree $d$ on $W(K)$ is a function $\eta\colon W(K) \rightarrow \bR$ which can be expressed as $\hat{f}\rest_{W(K)}$ for some extension $L/K$ and $f \in L[X]_d$.
  A \emph{virtual divisor} on $W(K)$ is a uniform limit of virtual divisors (of various degrees), i.e., in the distance:
  \begin{gather*}
    d(\eta,\theta) = \sup_{\xi \in W(K)} \bigl| \eta(\xi) - \theta(\xi) \bigr|.
  \end{gather*}
  When $\eta$ is a virtual divisor of degree $d$, we define its \emph{$d$-width} as
  \begin{gather*}
    w_d(\eta) = \sup \eta - \sup \bigl\{ \tilde{v}(f)/d : f \in L[X]_d, \ \eta = \hat{f}\rest_{W(K)} \bigr\},
  \end{gather*}
  as $L$ varies over all possible extensions of $K$.
\end{dfn}

By \autoref{prp:PolynomialTrace}, every virtual divisor is a bounded function, definable in the sense of \autoref{sec:Definability}.
The space of virtual divisors on $W(K)$ of a given degree is complete for uniform convergence.
Also, by an easy ultra-product argument, the second supremum in the definition of $w_d(\eta)$ is attained as a maximum.

In some sense, all virtual divisors are normalised to have ``degree one'' (one can imagine a more general definition, bur for our purposes only normalised virtual divisors are needed).
The degree of a virtual divisor, as per \autoref{dfn:VirtualDivisor}, measures its \emph{complexity}.
We have $\hat{f} = \widehat{f^m}$ so a virtual divisor of degree $d$, is also of degree $m d$ for every $m \geq 1$.

If $\eta$ is a virtual divisor on $W(K)$ and $L/K$ is an extension, then $\eta$ extends naturally to a virtual divisor on $W(L)$, by applying the same definition.
Alternatively, if $\eta = \hat{f}\rest_{W(K)}$ for $f \in M[X]_d$, then we may identify $f$ with a polynomial over $\Frac(L \otimes_K M)$, thus obtaining an extension of $\eta$ to $\hat{f}\rest_{W(L)}$.
By \autoref{prp:ValuedFieldIndependence}, we obtain the same extension either way.
Therefore, we may speak of a \emph{virtual divisor on $W$} which is \emph{defined over $K$}, which is the terminology we shall use from here on.
By quantifier elimination, if $\eta$ and $\theta$ are virtual divisors on $W$, say defined over $K$, then $d(\eta,\theta)$ is the same when calculated in $W(K)$ or in $W(L)$ for any extension $L/K$.

\begin{lem}
  \label{lem:VirtualDivisorOperations}
  Let $W \subseteq \bP^n$ be algebraic, defined over $K$.
  Then every constant function is a virtual divisor on $W$, of degree one.
  The family of virtual divisors on $W$ is closed under translation (by a real), minimum (of finitely many), uniform limit, and finite convex combinations.
  Moreover, all but convex combinations preserve the degree.
\end{lem}
\begin{proof}
  Closure under translation, minimum and uniform limit is clear from the definition, and the constant $0$ is $\min_{i<n} \widehat{X}_i\rest_{W(K)}$.
  We can calculate averages as
  \begin{gather*}
    \frac{\hat{f}\rest_{W(K)} + \hat{g}\rest_{W(K)}}{2} = \hat{h}\rest_{W(K)}, \qquad h = f^{\deg g} g^{\deg f}.
  \end{gather*}
  Closure under convex combinations follows by closure under uniform limits and boundedness.
\end{proof}

\begin{lem}
  \label{lem:VirtualDivisorDefinitionInequality}
  Let $L/K$ be an extension, and let $\eta = \hat{f}\rest_{W(K)}$, where $f \in L[X]_d$, be a virtual divisor.
  Then for every $\xi \in W(L)$ we have
  \begin{gather*}
    \eta(\xi) \leq \hat{f}(\xi).
  \end{gather*}
\end{lem}
\begin{proof}
  Let $M$ be a copy of $L$, let $g \in M[X]_d$ be the corresponding copy of $f$, and let us embed $L$ and $M$ in $\Frac(L \otimes_K M)$.
  There is a natural morphism $\varphi\colon L \otimes_K M \rightarrow L$, and by definition of the tensor product valuation, it satisfies $v_L\bigl( \varphi(z) \bigr) \geq (v_L \otimes v_M)(z)$.
  For $\xi = [x] \in W(L)$ we have $\eta(\xi) = \hat{g}(\xi)$ and $\varphi\bigl( g(x) \bigr) = f(x)$.
  Our assertion follows.
\end{proof}

\begin{dfn}
  \label{dfn:ChowFormDistance}
  Let $L/K$ be a valued field extension, and let $\fC$ and $\fC'$ be two Chow forms over $L$, both in same dimension $\ell$.
  We define the $K$-distance between $\fC$ and $\fC'$ as
  \begin{gather*}
    d_K(\fC,\fC') = \sup_F \left| \frac{v( F \wedge \fC )}{\deg F \deg \fC} - \frac{v( F \wedge \fC' )}{\deg F \deg \fC'} \right|,
  \end{gather*}
  as $F$ varies over all families of $(\ell+1)$ non-constant homogeneous polynomials over $K$, agreeing that $|\infty - \infty| = 0$.
\end{dfn}

From this point onward we are going to consider chains in projective space as coded by Chow forms.
We shall be using the wedge notation for the algebraic intersection of Chow forms with hypersurfaces, as in \cite{BenYaacov:Vandermonde}.

\begin{lem}
  \label{lem:ChowFormDistance}
  Let $L/K$ and $M/K$ be two valued field extensions, and let $W \subseteq \bP^n$ be defined over $K$
  Let $f$ and $g$ be two homogeneous non-constant polynomials over $L$, and let $\fC$ and $\fC'$ be two Chow forms over $L$, both in same dimension $\ell$, associated to subsets of $W$.
  Then, working in the free amalgam $\Frac(L \otimes_K M)$, we have
  \begin{gather*}
    d_K(f \wedge \fC, g \wedge \fC') \leq d(\hat{f}\rest_{W(K)},\hat{g}\rest_{W(K)}) + d_K(\fC,\fC').
  \end{gather*}
\end{lem}
\begin{proof}
  It will suffice to prove that
  \begin{gather*}
    d_K(f \wedge \fC, g \wedge \fC) \leq d(\hat{f}\rest_{W(K)},\hat{g}\rest_{W(K)}),
    \qquad
    d_K(f \wedge \fC, f \wedge \fC') \leq d_K(\fC,\fC').
  \end{gather*}
  We have already observed that, by \autoref{prp:ValuedFieldIndependence}:
  \begin{gather*}
    d(\hat{f}\rest_{W(K)},\hat{g}\rest_{W(K)}) = d(\hat{f}\rest_{W(M)},\hat{g}\rest_{W(M)}).
  \end{gather*}
  The resultant form (in some degrees) $F \mapsto F \wedge \fC$ is just a homogeneous polynomial in the coefficients of $\ell+1$ indeterminate polynomials (in those degrees), so by the same reasoning:
  \begin{gather*}
    d_K(\fC,\fC') = d_L(\fC,\fC').
  \end{gather*}
  If $H$ is any family of $\ell$ polynomials over $K$, then $H \wedge \fC$ splits as $\prod_{i<D} x_i$, where $D = \deg H \deg \fC$ and $[x_i] \in W(M)$.
  Therefore
  \begin{align*}
    d(\hat{f}\rest_{W(K)},\hat{g}\rest_{W(K)})
    = {} & d(\hat{f}\rest_{W(M)},\hat{g}\rest_{W(M)}) \\
    \geq {} & \frac{1}{D} \left| \frac{\sum_{i<D} v \circ f(x_i)}{\deg f} - \frac{\sum_{i<D} v \circ g(x_i)}{\deg g} \right| \\
    = {} & \left| \frac{v( H \wedge f \wedge \fC )}{\deg H \deg f \deg \fC} - \frac{v( H \wedge g \wedge \fC )}{\deg H \deg g \deg \fC} \right|.
  \end{align*}
  Similarly,
  \begin{gather*}
    d_K(\fC,\fC') = d_L(\fC,\fC')
    \geq \left| \frac{v( H \wedge f \wedge \fC )}{\deg H \deg f \deg \fC} - \frac{v( H \wedge f \wedge \fC' )}{\deg H \deg f \deg \fC'} \right|.
  \end{gather*}
  The two inequalities, and our assertion, follow.
\end{proof}

\begin{dfn}
  \label{dfn:VirtualChain}
  We define a \emph{virtual chain} of dimension $\ell$ inside $W$, over a given field $K$, by taking all Chow forms of dimension $\ell$ associated to subsets of $W$, defined over extensions of $K$, dividing by the kernel of $d_K$, and completing.
  The image of a Chow form $\fC$ will be denoted $\sC = \widehat{\fC}\rest_K$ (or just $\widehat{\fC}$, if $\fC$ is already over $K$).
  Assume that $\sC$ is a virtual chain of dimension $\ell$ inside $W$, and $\eta$ is a virtual divisor on $W$, both given as uniform limits
  \begin{gather*}
    \sC = \lim \widehat{\fC_k}\rest_K,
    \qquad
    \eta = \lim \hat{f}_k\rest_{W(K)},
  \end{gather*}
  where $\fC_k$ are over $L \supseteq K$ and $f_k$ are over $M \supseteq K$.
  Then we define
  \begin{gather*}
    \eta \wedge \sC = \lim \widehat{f_k \wedge \fC_k}\rest_K,
  \end{gather*}
  where $f_k \wedge \fC_k$ is calculated in $\Frac( L \otimes_K M)$.
  By \autoref{lem:ChowFormDistance}, this is a uniform limit, resulting in a virtual Chow form of dimension $\ell-1$ inside $W$, which only depends on $\eta$ and $\sC$.

  When iterating this with the same $\eta$, we may also write
  \begin{gather*}
    \eta^{\wedge k} \wedge \sC = \underset{k\ \text{times}}{\underbrace{\eta \wedge \ldots \wedge \eta}} \wedge \sC.
  \end{gather*}
  When $\sC = \widehat{\fC}_{\bP^n}$, for the canonically normalised Chow form of $\bP^n$, we omit it, just writing $\eta \wedge \ldots \wedge \theta$ or $\eta^{\wedge k}$.
  Similarly, the virtual chain $\widehat{g \wedge \fC_{\bP^n}}$ will be simply denoted $\hat{g}$, when there is no risk of ambiguity.
\end{dfn}

By \autoref{lem:ChowFormDistance},
\begin{gather*}
  d(\eta \wedge \sC, \theta \wedge \sD)| \leq d(\eta,\theta) + d(\sC,\sD),
  \qquad
  \eta \wedge \theta \wedge \sC = \theta \wedge \eta \wedge \sC.
\end{gather*}

\section{Sub-valuations}

Let us begin with a few general definitions.
Throughout, a \emph{ring} is commutative and unital.

\begin{dfn}
  \label{dfn:SubValuation}
  A \emph{sub-valuation} on a ring $A$ is a function $u\colon A \rightarrow \bR \cup \{\infty\}$ which satisfies the following properties:
  \begin{align*}
    u(ab) \geq {} & u(a) + u(b) && \text{(sub-multiplicative)} \\
    u(a^2) = {} & 2 u(a) && \text{(power-multiplicative)} \\
    u(a+b) \geq {} & u(a) \wedge u(b) && \text{(ultra-metric)} \\
    u(0) = {} & \infty.
  \end{align*}
  \begin{enumerate}
  \item We say that $u$ is a \emph{proper} sub-valuation if $u(1) = 0$.
  \item We say that $u$ is a \emph{valuation} if it is proper and multiplicative: $u(ab) = u(a) + u(b)$.
  \item We define $\ker u = \bigl\{ a \in A : u(a) = \infty \bigr\}$, and say that $u$ is \emph{reduced} if $\ker u = \{0\}$.
  \item If $A = \bigoplus A_m$ is a graded ring, and $u\bigl(\sum a_m\bigr) = \min u(a_m)$ whenever $\sum a_m$ is a decomposition into homogeneous components, then we say that $u$ is a \emph{homogeneous sub-valuation}.
  \end{enumerate}
\end{dfn}

Notice that $u \equiv \infty$ is the unique improper sub-valuation on $A$.
It follows easily from the axioms that $u(a^n) = n u(a)$ for all $n \in \bN$ (where $0^0 = 1$ and $0 \cdot \infty = 0$).

We consider general (sub-)valuations as generalisations, in the ``valued category'', of radical (prime) ideals on $A$.
In particular, a $\{0,\infty\}$-valued (sub-)valuation on $A$ carries exactly the same information as a radical (prime) ideal, namely its kernel.

\begin{fct}[see Bergman \cite{Bergman:WeakNullstellensatz}]
  \label{fct:WeakNullstellensatz}
  For every proper sub-valuation $u$ on a ring $A$ and $a \in A$ there exists a valuation $v \geq u$ on $A$ such that $v(a) = u(a)$.
  In other words,
  \begin{gather*}
    u = \inf \, \{ \text{valuation} \ v : v \geq u \}.
  \end{gather*}
\end{fct}

\begin{dfn}
  \label{dfn:GeneratedSubValuation}
  Let $A$ be a ring.
  A pair $(a,\alpha)$, where $a \in A$ and $\alpha \in \bR \cup \{\infty\}$, will be called a \emph{condition} (over $A$), which we may also denote informally as ``$u(a) \geq \alpha$''.
  When $A$ is graded and $a \in A$ is homogeneous, we shall say that $(a,\alpha)$ is a \emph{homogeneous condition}.

  Let $C \subseteq A \times \bigl( \bR \cup \{\infty\}\bigr)$ be a set of conditions such that $\pi(C) = \bigl\{ a \in A : \exists \alpha \ (a,\alpha) \in C \bigr\}$ generates $A$.
  We define the sub-valuation \emph{generated} by $C$, denoted $\langle C \rangle$, to be the least sub-valuation $u$ on $A$ satisfying $u(a) \geq \alpha$ for every condition $(a,\alpha) \in C$.

  For the purposes of this definition we identify a subset $C \subseteq A$ (which generates $A$) with the set of conditions $\bigl\{ (a,0) : a \in C \bigr\}$, so $\langle C \rangle$ is least such that $u\rest_C \geq 0$.
\end{dfn}

\begin{lem}
  \label{lem:GeneratedSubValuation}
  With the hypotheses of \autoref{dfn:GeneratedSubValuation}, the generated sub-valuation $u = \langle C \rangle$ exists, and can be recovered by (where $\min \emptyset = \infty$):
  \begin{gather*}
    u(a) = \sup \, \left\{ \frac{1}{n} \min_i \sum_j \alpha_{ij} : a^n = \sum_i \prod_j b_{ij}, \ \text{where} \ n \geq 1 \ \text{and} \ (b_{ij},\alpha_{ij}) \in C \right\}.
  \end{gather*}
  It is proper if and only if, for every $(b_{ij},\alpha_{ij}) \in C$:
  \begin{gather*}
    \sum_i \prod_j b_{ij} = 1 \qquad \Longrightarrow \qquad \min_i \sum_j \alpha_{ij} \leq 0.
  \end{gather*}
\end{lem}
\begin{proof}
  Let us show that $u$ satisfies the ultra-metric inequality.
  So let $a,b \in A$, and let $\rho < u(a) \wedge u(b)$ be arbitrary.
  Then there exist $n$ and $(c_{ij},\alpha_{ij})$ in $C$ such that $a^n = \sum_i \prod_j c_{ij}$ and $n\rho < \min_i \sum_j \alpha_{ij}$, and similarly $(d_{ij},\beta_{ij})$ for $b$.

  For each $0 \leq k < n$, we may express $a^k b^{n-k} = \sum_i \prod_j e_{kij}$, where $(e_{kij},\gamma_{kij}) \in C$.
  Let $\delta = \min_k \min_i \sum_j \gamma_{kij}$, and all we care about is that $\delta > -\infty$.
  For $\ell \in \bN$ we have
  \begin{gather*}
    (a+b)^{\ell n+n-1} = \sum_{0 \leq m \leq \ell} \left( a^{nm}b^{n(\ell-m)} \sum_{k < n} N_{m,k} a^k b^{n-k} \right),
  \end{gather*}
  where $N_{m,k}$ is some binomial coefficient.
  In each each term, express $a^n$ using the $c_{ij}$, express $b^n$ using the $d_{ij}$, and the term $a^k b^{n-k}$ using the $e_{kij}$.
  We obtain that
  \begin{gather*}
    u(a+b)
    \geq
    \frac{\ell n \rho + \delta}{\ell n + n + 1}.
  \end{gather*}
  Letting $\ell \rightarrow \infty$, we obtain $u(a+b) \geq \rho$, as desired.

  Everything else is easy.
\end{proof}

Every partial map $u'\colon A \dashrightarrow \bR \cup \{\infty\}$ can be identified with its graph, which is a set of conditions, so, assuming that $\dom u'$ generates $A$, we may speak of the generated sub-valuation $\langle u' \rangle$.
In particular, a sub-valuation $u$ generates itself.
If $A$ is graded and $C$ consists solely of homogeneous conditions, then $\langle C \rangle$ is homogeneous.
A sub-valuation $u$ is homogeneous if and only if it is generated by its restriction to homogeneous elements.

\begin{dfn}
  \label{dfn:AlgebraSubValuation}
  Let $(B,u_B)$ be a sub-valued ring, and $\varphi\colon B \rightarrow A$ a ring morphism, making $A$ a $B$-algebra.
  \begin{enumerate}
  \item
    A \emph{$B$-sub-valuation} on $A$ is a sub-valuation $u_A$ which satisfies, in addition, $u_A \circ \varphi \geq u_B$ (i.e., $u_A\bigl( \varphi(b)\bigr) \geq u_B(b)$ for all $b \in B$).
  \item
    Let $C$ be a set of conditions over $A$ (in the sense of \autoref{dfn:GeneratedSubValuation}) such that $\pi(C)$ generates $A$ as a $B$-algebra.
    Then the \emph{$B$-sub-valuation generated by $C$}, denoted $\langle C \rangle_B$, is the sub-valuation on $A$ generated by $C \cup \bigl\{ \bigl( \varphi(b),u_B(b) \bigr) : b \in B\bigr\}$.
  \end{enumerate}
\end{dfn}

Notice that if $B = K$ is a valued field, then any $K$-sub-valuation on $A$ is, in particular, a $K$-vector space valuation, and must agree with $v_K$ on the image of $K$ in $A$.

\begin{dfn}
  \label{dfn:AlgebraSubValuationUniformConvergence}
  Assume that $A$ is a finitely generated $B$-algebra.
  Fix a finite generating tuple $c \in A^k$.
  For $a \in A$, define $\deg_c a$ to be the least degree of a polynomial $f \in B[X]$ such that $a = f(c)$.
  We say that a sequence of $B$-sub-valuations $u_m$ \emph{converges uniformly} to $u$, if for every $\varepsilon > 0$ there exists $N$ such that for all $m \geq N$ and $a \in A$:
  \begin{gather*}
    \bigl| u(a) - u_m(a) \bigr| \leq \varepsilon \deg_c a, \qquad \text{where}\ |\infty - \infty| = 0.
  \end{gather*}
\end{dfn}

It is easy to check that this does not depend on the choice of generating tuple.

\begin{dfn}
  \label{dfn:AlmostFinitelyGeneratedSubValuation}
  Continuing \autoref{dfn:AlgebraSubValuation}, assume that $A$ is a finitely generated $B$-algebra.
  Let $u$ be a $B$-sub-valuation on $A$.
  \begin{enumerate}
  \item
    We say that $u$ is \emph{finitely generated as a $B$-sub-valuation} (or \emph{finitely generated over $B$}) if is generated, as a $B$-sub-valuation, by finitely many conditions.
    Notice that this implies, in particular, that $A$ is finitely generated as a $B$-algebra.
  \item
    We say that $u$ is \emph{almost finitely generated as a $B$-sub-valuation} (or \emph{almost finitely generated over $B$}) if it is a uniform limit of a sequence of finitely generated $B$-sub-valuations.
  \item
    \label{item:AlmostFinitelyGeneratedSubValuationBySet}
    More precisely, let $C = \bigl\{ (a_i,\alpha_i) : i \in \bN\bigr\}$ be a countable family of conditions over $A$, such that $\pi(C) = \{a_i : i \in \bN\}$ generates $A$ over $B$.
    Then $\{a_i : i < N\}$ generates $A$ over $B$ for some $N$, and we say that $C$ \emph{almost finitely generates} $u$ as a $B$-sub-valuation if
    \begin{gather*}
      \bigl\langle (a_i,\alpha_i) : i < N + n \bigr\rangle \underset{n \rightarrow \infty}\longrightarrow u \qquad \text{uniformly.}
    \end{gather*}
  \end{enumerate}
\end{dfn}

Notice that:
\begin{enumerate}
\item \autoref{dfn:AlmostFinitelyGeneratedSubValuation}\autoref{item:AlmostFinitelyGeneratedSubValuationBySet} does not depend on the choice of enumeration of the conditions.
\item A $B$-sub-valuation on $A$ is (almost) finitely generated if and only if it (almost finitely) generated by its restriction to a finite (countable) subset of $A$.
\item It is homogeneous if and only if the finite (countable) generating set can be taken to consist of homogeneous elements.
\end{enumerate}

Let us now restrict the scope a little.
As usual, we let $W \subseteq \bP^n$ be algebraic, defined over $K \vDash ACMVF$.
We let $I(W)$ denote the radical homogeneous ideal associated to $W$, and let $K[W] = K[X]/I(W)$, which inherits the structure of a graded $K$-algebra (it is the graded $K$-algebra of regular functions on the associated line bundle).
Notice that any reduced finitely generated graded $K$-algebra can be presented in this manner, so even from a purely algebraic point of view this is essentially the general case.

\begin{rmk}
  \label{rmk:NoHBT}
  A $K$-sub-valuation on $K[X]$ need not necessarily be almost finitely generated.
  In other words, we do not expect to have an analogue of the Hilbert Basis Theorem regarding Noetherianity.
  Indeed, assume that $(a_i : i \in \bN) \subseteq K$ are such that $v(a_i - a_j) = 0$ for all $i \neq j$.
  Then the sub-valuation $u$ generated by the conditions $u\bigl( \prod_{i<m} (X-Ya_i) \bigr) \geq m^2$ is not almost finitely generated.
\end{rmk}

The algebra $K[X]$ has a canonical set of generators over $K$, giving rise to a canonical notion of distance on proper sub-valuations:
\begin{gather*}
  d(u,w) = \sup_f \frac{| u(f) - w(f)|}{\deg f},
\end{gather*}
as $f$ varies over all non-constant homogeneous polynomials (again, $|\infty - \infty| = 0$).
Then $u_n \rightarrow u$ uniformly, in the sense of \autoref{dfn:AlgebraSubValuationUniformConvergence}, if and only if $d(u_n,u) \rightarrow 0$.
If $u$ and $w$ have distinct kernels then clearly $d(u,w) = \infty$.
As we show below, if $u$ and $w$ are homogeneous and almost finitely generated over $K$ then the converse holds as well.

\begin{exm}
  \label{ntn:PlainValuation}
  The simplest valuation on $K[X]$ is $\tilde{v}$.
  It is homogeneous, and (finitely) generated by the conditions $\tilde{v}(X_i) \geq 0$.
\end{exm}

Conversely (but this is far from being the general case):

\begin{prp}
  \label{prp:LinearlyGeneratedSubValuation}
  Assume $u$ is a reduced $K$-sub-valuation on $K[X]$ which is generated by a family $C$ of \emph{linear} conditions, i.e., of the form $(\lambda,\alpha)$ (or $u(\lambda) \geq \alpha$), where $\lambda \in K[X]_1$.
  Then up to a change of coordinates by a triangular matrix, $u$ is arbitrarily close to $\tilde{v}$.
  Consequently, $u$ is an almost finitely generated valuation.
\end{prp}
\begin{proof}
  By \autoref{lem:VectorSpaceValuationBasis}.
\end{proof}

Hilbert's Nullstellensatz can be stated as a bijection between (homogeneous) radical ideals and (projective) zero-sets.
The analogous statement in our setting asserts the existence of an isometric bijection between almost finitely generated homogeneous sub-valuations $u$ on $K[X]$ (analogous to homogeneous radical ideals) and projective zero-sets $W \subseteq \bP^n$ augmented with a virtual divisor $\eta$.
Moreover, $W$ is always the zero-set of $\ker u$ (this is just Hilbert's Nullstellensatz).

\begin{dfn}
  \label{dfn:Duals}
  Let $W \subseteq \bP^n$ be Zariski closed, and $\eta\colon W \rightarrow \bR$ any function.
  For homogeneous $f \in K[X]_m$ we define
  \begin{gather*}
    \eta^*_m(f) = \inf_{\xi \in W} \, \bigl( vf(\xi) - m \eta(\xi) \bigr) = m \inf_{\xi \in W} \, \bigl( \hat{f}(\xi) - \eta(\xi) \bigr).
  \end{gather*}
  We extend this to non-homogeneous polynomials by
  \begin{gather*}
    \eta^*(f) = \min_m \eta^*_m(f_m),
  \end{gather*}
  where $f = \sum_m f_m$ is the decomposition into homogeneous components.

  Conversely, let $u\colon K[W] \rightarrow \bR \cup \{\infty\}$ be any function.
  For $\xi \in W$ we define
  \begin{gather*}
    u^*(\xi) = \inf_f \, \left( \hat{f}(\xi) - \frac{u(f)}{\deg f} \right),
  \end{gather*}
  as $f$ varies over non-constant homogeneous polynomials.
\end{dfn}

\begin{thm}[Virtual divisor Nullstellensatz]
  \label{thm:Nullstellensatz}
  Let $W \subseteq \bP^n$ be Zariski closed, and let $u$ be a reduced homogeneous sub-valuation on $K[W]$, almost finitely generated, say by the homogeneous conditions $C = \bigl\{ (f_k,\gamma_k) : k \in \bN \bigr\}$.
  \begin{enumerate}
  \item We have
    \begin{gather*}
      u^* = \inf_k \left( \hat{f_k} - \frac{\gamma_k}{\deg f_k} \right)
    \end{gather*}
    uniformly, i.e.,
    \begin{gather*}
      \min_{k < m} \left( \hat{f_k} - \frac{\gamma_k}{\deg f_k} \right) \underset{m \rightarrow \infty}\longrightarrow u^* \qquad \text{uniformly}.
    \end{gather*}
    In particular, $u^*$ is a virtual divisor on $W$.
    Conversely, every virtual divisor on $W$ arises in this manner, and moreover, with all $\gamma_k = 0$.
  \item The double dual $u^{**}$ is equal to $u$.
    Consequently, for any virtual divisor $\eta$ on $W$ we have $\eta^{**} = \eta$.
  \item This bijection between virtual divisors on $W$ and reduced almost finitely generated homogeneous sub-valuations on $K[W]$ is isometric:
    \begin{gather*}
      d(u,w) = d(u^*,w^*).
    \end{gather*}
  \end{enumerate}
\end{thm}
\begin{proof}
  Let
  \begin{gather*}
    u_m = \bigl\langle (f_k,\gamma_k) : k < m \bigr\rangle_K,
    \qquad
    \eta_m = \min_{k<m} \left( \hat{f_k} - \frac{\gamma_k}{\deg f_k} \right).
  \end{gather*}
  Clearly, $u_m^* \leq \eta_m$, and for $m$ large enough, the sequence $(f_k : k < m)$ generates $K[W]$ as a $K$-algebra, and its linear members generate $K[W]_1$ as a $K$-vector space.
  If $\xi \in W$ and $u_m^*(\xi) < \rho$, then there exists a homogeneous polynomial $f$ which does not vanish at $\xi$, such that $\frac{u(f)}{\deg f} > \hat{f}(\xi) - \rho$.
  The same remains true if we replace $f$ with $f^n$ (for $n \geq 1$), so we may assume that
  \begin{gather*}
    f = \sum_i \prod_j g_{ij}, \qquad \frac{\min_i \sum_j \delta_{ij}}{\deg f} + \rho > \hat{f}(\xi),
  \end{gather*}
  where $(g_{ij},\delta_{ij}) \in \bigl\{ (f_k,\gamma_k) : k < m\bigr\}$.
  But then, unwinding the definition of $\hat{f}$, there must exist a pair $(f_k,\gamma_k) = (g_{ij},\delta_{ij})$ such that $\frac{\gamma_k}{\deg f_k} + \rho > \hat{f_k}(\xi)$, so $\rho > \eta_m(\xi)$, and $u_m^* = \eta_m$.
  It follows immediately from the definitions that $d(u^*,w^*) \leq d(u,w)$, so $\eta_m \rightarrow u^*$ uniformly.
  Since the $v(K^\times)$ is dense in $\bR$, we may take all the $\gamma_k$ to be zero.
  By \autoref{prp:PolynomialTrace}, each $\eta_m$ is a virtual divisor, and so is $u^*$, and conversely, every virtual divisor arises in this fashion, proving the first item.

  For the rest, let us indeed assume that $\gamma_k = 0$ for all $k$.
  It is immediate from \autoref{dfn:Duals} that $u^{**}$ is a homogeneous sub-valuation on $K[W]$, and that $u^{**} \geq u$.
  Assume that for some $g \in K[W]_d$ we have $u(g) < \rho$.
  By \autoref{fct:WeakNullstellensatz}, there exists a valuation $w$ on $K[W]$ such that $w \geq u$ and $w(g) < \rho$.
  Then $\ker w$ is a prime ideal of $K[W]$ containing $I(W)$, and $w$ induces a valuation $v_L$ on the field $L = \Frac\bigl( K[X] / \ker w \bigr)$, extending $v_K$.

  Let $\xi = [x] \in \bP^n(L)$, where $x \in L^{n+1}$ is the image of $X$.
  Then $v \circ f_k(x) = w(f) \geq u(f) \geq 0$, $\xi \in W(L)$, and $u^*$ extends to $W(L)$ by definability, so:
  \begin{gather*}
    u^*(\xi)
    = \inf_k \hat{f_k}(\xi)
    = \inf_k \frac{w(f_k)}{\deg f_k} - \tilde{v}(x)
    \geq - \tilde{v}(x).
  \end{gather*}
  Therefore
  \begin{gather*}
    vg(\xi) - d u^*(\xi)
    \leq vg(\xi) + d \tilde{v}(x)
    = w(g) < \rho.
  \end{gather*}
  By quantifier elimination in ACMVF, there exists $\zeta \in W(K)$ such that
  \begin{gather*}
    vg(\zeta) - d u^*(\zeta) < \rho,
  \end{gather*}
  so $u^{**}(g) < \rho$ as well, proving that $u^{**} = u$.
  Since every virtual divisor $\eta$ on $W$ can be expressed as $u^*$, we also have $\eta^{**} = u^{***} = u^* = \eta$, proving the second item.

  It is again immediate from the definitions that $d(\eta^*,\theta^*) < d(\eta,\theta)$ for any two virtual divisors.
  Therefore $d(u,w) = d(u^{**},w^{**}) \leq d(u^*,w^*) \leq d(u,w)$, completing the proof.
\end{proof}

Let us recall an easy corollary of the classical Nullstellensatz (we prefer to work in projective space, but one can also give an affine version).

\begin{fct}
  \label{fct:ClassicalNullstellensatzUniformGeneration}
  For every $n,N \in \bN$ there exists $m \in \bN$ such that the following holds.
  Let $K$ be an algebraically closed field, let $I \subseteq K[X]$ be homogeneous, generated by polynomials of degree at most $N$, and let $g$ be homogeneous of degree at most $N$ vanishing on $V(I) \subseteq \bP^n(K)$.
  Then $g^m \in I$.
\end{fct}
\begin{proof}
  Let $D = \dim K[X]_{\leq N}$, which only depends on $n$ and $N$.
  Then one can generate $I$ with fewer than $D$ polynomials, all of degree at most $N$.
  In particular, the fact that $g$ vanishes on $V(I)$ can be expressed in first-order logic.

  Assume that no such $m$ exists.
  Taking an ultra-product of counter-examples for all $m$, we obtain an algebraically closed field $K$, a finite family of homogeneous polynomials of degree at most $N$ generating an ideal $I$, and $g$ homogeneous of degree at most $N$, which vanishes on $V(I)$.
  By Hilbert's Nullstellensatz, there exists $M$ such that $g^M \in I$.
  Viewed as a property of the coefficients, the same must hold in infinitely many of the aforementioned counter-examples.
  In particular, it holds in the $m$th counter-example for some $m \geq M$, a contradiction.
\end{proof}

The corresponding corollary of \autoref{thm:Nullstellensatz} is the following.

\begin{cor}
  \label{cor:NullstellensatzUniformGeneration}
  For every $n,N \in \bN$ there exists $m \in \bN$ such that the following holds.
  Let $K$ be an algebraically closed valued field, and let $W \subseteq \bP^n(K)$ be defined by polynomials of degree at most $N$.
  Let $u$ be a reduced homogeneous sub-valuation on $K[W]$, almost finitely generated by a family $C$ of homogeneous conditions $(f,\gamma)$, where $\deg f \leq N$.
  Let $c = \sup \ \biggl\{ \bigl| \gamma - \tilde{v}(f) \bigr| : (f,\gamma) \in C\biggr\}$.
  Then for every $g$ homogeneous of degree at most $N$:
  \begin{gather}
    \label{eq:NullstellensatzUniformGeneration}
    u(g) - \frac{c}{N}
    \leq \frac{1}{m} \sup \, \left\{ \min_i v(a_i) + \sum_j \gamma_{ij} : g^m = \sum_i a_i \prod_j f_{ij}, \ \text{where} \ a_i \in K \ \text{and} \ (f_{ij},\gamma_{ij}) \in C \right\}
    \leq u(g).
  \end{gather}
\end{cor}
\begin{proof}
  The second inequality always holds, so we only prove the first.
  Possibly re-scaling the valuation, we may assume that $c = 1$.
  As in the proof of \autoref{fct:ClassicalNullstellensatzUniformGeneration}, let $D = \dim K[X]_{\leq N}$.
  Then a sub-valuation arbitrarily close to $u$ can be generated by a subset of $C$ of size at most $D$, so we may add the hypothesis that $|C| \leq D$.
  Similarly, we require that $\tilde{v}(f) = 0$, and so $\gamma \in [-1,1]$, for every $(f,\gamma) \in C$, and that $\tilde{v}(g) = 0$.

  By \autoref{thm:Nullstellensatz}, value $u(g)$ is uniformly definable in $g$ and $C$ (i.e., in the coefficients of $g$, viewed as a point in $\bP$, and, for each $(f,\gamma) \in C$, in $f$, viewed in the same manner, and in the real constant $\gamma$).
  Letting $u_m(g)$ denote the middle expression in \autoref{eq:NullstellensatzUniformGeneration}, it similarly definable in $g$ and $C$.

  Assume that our assertion is false, and again take and ultra-product of counter-examples.
  We may assume that for every $m$, the set of all multiples of $m$ is large in the ultra-filter.
  Since $u_m$ increases with the divisibility relation, we obtain $g$ and $C$ such that $u(g) \geq u_m(g) + 1/N$ for all $m$, which is impossible by definition of $u(g)$.
\end{proof}

\begin{cor}
  \label{cor:NullstellensatzFieldExtension}
  Let $\eta_K$ be a virtual divisor on $W(K)$, and let $u_K = \eta_K^*$ be the corresponding sub-valuation on $K[W]$.
  Let $L/K$ be an extension.
  Let $\eta_L$ be the extension of $\eta_K$ to $W(L)$ by definability, and let $u_L = \langle u_K \rangle_L$ be the sub-valuation on $L[W]$ generated by $u_K$.
  Then $\eta_L^* = u_L = u_K \otimes v_L$.
  In other words, the duality commutes with scalar extensions.
\end{cor}
\begin{proof}
  Let $u_K$ be almost finitely generated by $C = \bigl\{ f_k : k \in \bN \bigr\}$ over $K$, with $f_k$ homogeneous, so $\eta_K = \inf_k \, \hat{f_k}\rest_{W(K)}$ uniformly.
  Then $u_L$ is almost finitely generated by $C$ over $K$, and $\eta_L = \inf_k \, \hat{f_k}\rest_{W(L)}$ uniformly, so indeed $u_L = \eta_L^*$.

  It is also clear that $u_K \otimes v_L \leq u_L$, since being a $K$-algebra sub-valuation is a stronger property than begin a $K$-vector space valuation.
  If they are not equal, then there exists some $g \in L[W]_d$ and $\varepsilon > 0$ such that $(u_K \otimes v_L)(g) + 2 \varepsilon < u_L(g)$.
  We can find a basis $(f_i : i < N)$ for $K[W]_d$, with respect to which $u_K$ is $\varepsilon$-close to $\tilde{v}$.
  This is also a basis for $K[W]_d \otimes_K L$, with respect to which $u_K \otimes v_L$ is $\varepsilon$-close to $\tilde{v}$.
  We can express $g = \sum b_i f_i$, with $b_i \in L$ so $u_L(g) > (u_K \otimes v_L)(g) + 2 \varepsilon \geq \tilde{v}(b) + \varepsilon$.
  In other words, in $L$ we have:
  \begin{gather*}
    \sup_{[z] \in \bP^{N-1}} \, \inf_{[x] \in W} \Bigl( v\bigl( \sum z_i f_i(x) \bigr) - d \tilde{v}(x) - \tilde{v}(z) \Bigr) > \varepsilon.
  \end{gather*}
  By quantifier elimination, the same holds in $K$.
  In other words, there are $a \in K^N$ such that, if $h = \sum a_i f_i \in K[W]_d$, then $u_K(h) > \tilde{v}(a) + \varepsilon$, contradicting the choice of $(f_i)$.
  (Notice that this implies, in particular, that the natural surjective map $K[W] \otimes_K L \rightarrow L[W]$ is bijective.)
\end{proof}

\begin{cor}
  \label{cor:NullstellensatzRestriction}
  Let $W \subseteq V \subseteq \bP^n$ be algebraic sets and $\eta$ a virtual divisor on $V$.
  Then for homogeneous $f \in K[X]$ we have
  \begin{gather}
    \label{eq:NullstellensatzRestriction}
    (\eta\rest_W)^*(f)
    = \lim_{m \rightarrow \infty} \, \frac{\eta^*\bigl( f^m + I(W) \bigr)}{m}.
  \end{gather}
\end{cor}
\begin{proof}
  Say that $\eta = \inf_k \, \hat{f_k}\rest_V$ uniformly, so $\eta^*$ is the least sub-valuation on $K[V]$ satisfying $\eta^*(f_k) \geq 0$.
  Let $u$ denote the right hand side, extended to non-homogeneous polynomials in the usual way.
  It is easy to check that the limit is increasing with respect to the divisibility relation $m \mid m'$, and to deduce that $u$ is indeed a homogeneous sub-valuation on $K[W]$.
  Pulled back to a sub-valuation on $K[V]$, it is clearly the least one satisfying $u \geq \eta^*$ and $\ker u \supseteq I(W)$.
  In other words, $u$ is the least sub-valuation on $K[W]$ satisfying $u(f_k) \geq 0$ for all $k$.
  On the other hand, $\eta\rest_W = \inf_k \, \hat{f_k}\rest_W$, so $u = (\eta\rest_W)^*$.
\end{proof}

\section{Statement of the main result (and a few easy cases)}
\label{sec:MainResult}

In this section we state our main result, \autoref{thm:MainResult}.
It relates the self-intersection of a virtual divisor $\eta$ on a projective variety, on the one hand, with the asymptotic behaviour of the volume of the dual $\eta^*_m$, on the other hand, as the degree $m$ goes to infinity.

Let $E$ be a vector space.
For any $k$, the symmetric group $\fS_k$ acts naturally on $E^{\otimes k}$, and we may define the exterior power of $E$ by
\begin{gather*}
  {\bigwedge}^k E = \left\{ x \in E^{\otimes k} : \sigma(x) = (\sgn \sigma) x \ \text{for all} \ \sigma \in \fS_k \right\}.
\end{gather*}
In particular, for a finite family of vectors $\bx \in E^k$ we shall write
\begin{gather*}
  \bx^\wedge = x_0 \wedge \cdots \wedge x_{k-1} = \sum_{\sigma \in \fS_k} (\sgn \sigma) x_{\sigma(0)} \otimes \cdots \otimes x_{\sigma(k-1)} \in {\bigwedge}^k E.
\end{gather*}
Given a valuation $u$ on $E$, this gives rise to a value $u^{\otimes k}(\bx^\wedge)$.
When there is no risk of ambiguity, we may simply denote this by $u(\bx^\wedge)$.

\begin{dfn}
  \label{dfn:Volume}
  When $\bx$ is a basis of $E$, we call $u(\bx^\wedge)$ the \emph{volume} of $(E,u)$ (or of $u$), relative to $\bx$, denoted $\vol_\bx E$ or $\vol_\bx u$, depending on context.
\end{dfn}

The volume can be viewed as a (particularly simple) special case of \autoref{dfn:VirtualChain}.
Indeed, we may assume that $E = K^m$, equipped with the standard basis $\be$.
Then $\eta\bigl( [x] \bigr) = u(x) - \tilde{v}(x)$ is a virtual divisor on $\bP^{m-1}$, and
\begin{gather*}
  \eta^{\wedge m} = u(\be^\wedge).
\end{gather*}

In order to state our main result we need one last technical definition.

\begin{dfn}
  \label{dfn:VarietyToHypersurfaceGood}
  Let $K_0$ be a field, $T$ an infinite family of indeterminates, and $L \supseteq K_0(T)$.
  We shall call a valuation on $K_0(T)$, or any larger field, \emph{good} (with respect to $K_0,T$), if
  \begin{enumerate}
  \item
    \label{item:VarietyToHypersurfaceGoodGeneric}
    either the tuple $T$ is $v$-generic over $K_0$,
  \item
    \label{item:VarietyToHypersurfaceGoodPAdic}
    or the restriction to $K_0(T)$ is the $P$-adic valuation for some irreducible polynomial $P \in K_0[T]$.
  \end{enumerate}
\end{dfn}

\begin{thm}
  \label{thm:MainResult}
  Let $K_0$ be an algebraically closed field, $T$ an infinite family of indeterminates, and $K = K_0(T)^a$.
  Let $d \in \bN$, and let $W \subseteq \bP^n(K)$ be a projective variety defined over $K_0$, of dimension $\ell$, with Chow form $\fC_W$.
  Then there exist a tuples $a$ in $K$, a sequence $\gamma_m \rightarrow 0$, and for infinitely many $m \in \bN$ there exists a basis $B_m$ for $K[W]_m$, such that for any good valuation on $K$, and any virtual divisor $\eta$ on $W(K)$ of degree $d$, defined over $K_0$, with $d$-width $w_d(\eta)$ as per \autoref{dfn:VirtualDivisor}, we have
  \begin{gather}
    \label{eq:MainResult}
    \frac{\vol_{B_m} \eta^*}{m \dim K[W]_m} \geq - \frac{\eta^{\wedge \ell+1} \wedge \widehat{\fC}_W}{\ell+1} + \gamma_m \bigl( \tilde{v}(a) - w_d(\eta) \bigr).
  \end{gather}
\end{thm}

Let us start with a few easy observations.

\begin{lem}
  \label{lem:MainResultObservations}
  Let $K$ be a valued field, $W \subseteq \bP^n$ a variety of dimension $\ell$, and let $\eta$ and $\theta$ be two virtual divisors on $W$.
  Let $\fC_W$ be a Chow form for $W$, and let $B$ be a basis for $K[W]_m$.
  Then
  \begin{gather}
    \label{eq:MainResultObservationsDistance}
    \left| \frac{\vol_B \eta^*}{m \dim K[W]_m} - \frac{\vol_B \theta^*}{m \dim K[W]_m} \right| \leq d(\eta,\theta), \qquad
    \left| \frac{\eta^{\wedge \ell+1} \wedge \widehat{\fC}_W}{\ell+1} - \frac{\theta^{\wedge \ell+1} \wedge \widehat{\fC}_W}{\ell+1} \right| \leq d(\eta,\theta).
  \end{gather}
  If $\theta = \eta + \alpha$ for some $\alpha \in \bR$, then
  \begin{gather}
    \label{eq:MainResultObservationsTranslation}
    \frac{\vol_B \theta^*}{m \dim K[W]_m} = \frac{\vol_B \eta^*}{m \dim K[W]_m} - \alpha, \qquad
    \frac{\theta^{\wedge \ell+1} \wedge \widehat{\fC}_W}{\ell+1} = \frac{\eta^{\wedge \ell+1} \wedge \widehat{\fC}_W}{\ell+1} + \alpha.
  \end{gather}
  Consequently, it will suffice to prove \autoref{thm:MainResult} in the case where $\sup \eta = 0$.
\end{lem}
\begin{proof}
  If $g \in K[W]_m$, then $\bigl| \eta^*(g) - \theta^*(g) \bigr| \leq m d(\eta,\theta)$.
  Therefore, if $z \in K[X]_m^{\otimes \dim K[X]_m}$ (and in particular if $z = B^\wedge$), then $|\eta^*(z) - \theta^*(z)| \leq m \dim K[W]_m d(\eta,\theta)$.
  The second inequality in \autoref{eq:MainResultObservationsDistance} is just \autoref{lem:ChowFormDistance}.
  If $\theta = \eta + \alpha$, then $\theta_m^* = \eta_m^* - m\alpha$, whence the identity in \autoref{eq:MainResultObservationsTranslation}.
  We may express $\eta = \hat{f}\rest_{W(K)}$ and $\theta = \widehat{af}\rest_{W(K)}$, where $v(a) = \alpha \deg f$, and the second identity follows.

  Let $\eta = \hat{f}\rest_{W(K)}$, where $f \in L[X]_d$ for some extension $L/K$, such that $\tilde{v}(f)$ is maximal possible.
  By \autoref{eq:MainResultObservationsTranslation} we may reduce to the case where $\sup \eta = 0$.
\end{proof}

We are going to prove \autoref{thm:MainResult} by calculating the volume of $K[X]_m$ with respect to particularly convenient bases.
For this we require a few basic properties of the volume function.

\begin{lem}
  \label{lem:Volume}
  Let $E$ be a valued vector space over $K$, and let $\bx \in E^m$ be a basis.
  \begin{enumerate}
  \item
    \label{item:VolumeBasisChange}
    If $\by = \bx A$ is any other basis (we consider bases as rows of vectors) then $\vol_\by E = \vol_\bx E + v(\det A)$.
  \item
    \label{item:VolumeScalarExtension}
    Let $L/K$ be an extension, and $E_L = E \otimes_K L$ be the extension of scalars.
    Then, identifying $\bx$ with $\bx \otimes 1 \in E_L^m$, the volume remains unchanged: $\vol_\bx E = \vol_\bx E_L$.
  \item
    \label{item:VolumeExplicit}
    We have
    \begin{gather*}
      \vol_\bx E = \sum_{i<m} u( x_i + F_i ), \qquad F_i = \Span(x_j : j < i).
    \end{gather*}
  \item
    \label{item:VolumeQuotient}
    Let $\by \subseteq \bx$ generate $F \subseteq E$.
    Then
    \begin{gather*}
      \vol_\bx E = \vol_\by F + \vol_{\bx \setminus \by} (E/F),
    \end{gather*}
    where by $\bx \setminus \by$ we actually mean the image in $E/F$.
  \end{enumerate}
\end{lem}
\begin{proof}
  \begin{enumerate}
  \item Since $\by^\wedge = (\det A) \bx^\wedge$.
  \item We may identify $(E_L)^{\otimes m}$ with $(E^{\otimes m})_L$ as valued vector spaces over $L$ (by reduction to $\tilde{v}$), and then $(\bx \otimes 1)^\wedge = \bx^\wedge \otimes 1$.
  \item By (the proof of) \autoref{lem:VectorSpaceValuationBasis}, together with the fact that the volume of $\tilde{v}$ with respect to the standard basis is zero.
  \item Follows from the previous item.
  \end{enumerate}
\end{proof}

Now to the construction of a basis, given a virtual divisor.

\begin{lem}
  \label{lem:VirtualDivisorSetup}
  Let $\eta$ be a virtual divisor of degree $d$ on $\bP^n$, defined over $K$, with $\sup \eta = 0$ and $\inf \eta = -\alpha$.
  Let $m \geq n$ be given.
  \begin{enumerate}
  \item
    \label{item:VirtualDivisorSetupExists}
    In an extension $L/K$, one can find $f \in L[X]_d$ weakly $v$-generic over $K$ such that $\eta = \hat{f}\rest_{\bP^n(K)}$, in which case $\tilde{v}(f) = -d\alpha$, so $\alpha = w_d(f)$.
    Moreover, one can construct a family $F = (f_i : i < m) \subseteq L[X]_d$ of polynomials are isomorphic to $f$ over $K$, and are independent over $K$ in the sense of \autoref{sec:Definability}, i.e., such that if $L_i = K(f_i)$ is the subfield generated by $K$ and the coefficients of $f_i$, then their compositum is $\Frac(L_0 \otimes_K \cdots \otimes_K L_{m-1})$, as a valued field.
  \item
    \label{item:VirtualDivisorSetupGeneric}
    Let us fix such a family.
    Then each $f_i$ is weakly $v$-generic over $K(f_j : i \neq j)$.
  \item
    \label{item:VirtualDivisorSetupGood}
    Let us also fix $dX \in L^{n+1}$ $v$-generic over $K(F)$, which we use as direction for Hasse derivation, as in \cite[Section~4]{BenYaacov:Vandermonde}.
    Then $dX$ is algebraically generic over $K(F)$, and $F$ is good in the sense of \cite[Section~4]{BenYaacov:Vandermonde}.
  \end{enumerate}
\end{lem}
\begin{proof}
  For \autoref{item:VirtualDivisorSetupExists}, choose $f$ such that $\eta = \hat{f}\rest_{\bP^n(K)}$.
  Adding $v$-generic elements (over $K$) to the coefficients of $f$ does not change $\eta$ and makes $f$ weakly $v$-generic.
  Since the variety here is the entire space, it is easy to see that $\tilde{v}(f) = d \inf \eta$.
  The construction of a field containing an independent family is essentially given in the statement.
  Items \autoref{item:VirtualDivisorSetupGeneric} and \autoref{item:VirtualDivisorSetupGood} are immediate.
\end{proof}

\begin{lem}
  \label{lem:VirtualDivisorBasis}
  Continuing \autoref{lem:VirtualDivisorSetup}, assume that $h \in K[X]_D$ is an additional polynomial, weakly $v$-generic over the prime field, so $Fh$ is again a good family (possibly $D = 0$ and $h = 1$).
  Let $\psi \subseteq \bigl[ (Fh)^{\partial n} \bigr]$ be a good set.
  Then, identifying $K[X]_{md+D-n}$ with $K^{\binom{md+D}{n}}$ via the basis $\Phi_{Fh,\psi}$, we have $\eta^* = \tilde{v} + O\bigl( \alpha - \tilde{v}(h) \bigl)$.
  In other words, there exists a constant $c$ which depends only on $n$, $d$ and $D$ (but not on $m$, $F$, or $h$) such that for all choices of coefficients $a_\xi \in L$:
  \begin{gather*}
    \left| \eta^*\left( \sum_{\xi \in \psi} a_\xi \varphi_{Fh,\psi,\xi} \right) - \min_{\xi \in \psi} v(a_\xi) \right| \leq c\bigl( \alpha - \tilde{v}(h) \bigr).
  \end{gather*}
\end{lem}
\begin{proof}
  Let us estimate $\eta^*(\varphi_{Fh,\psi,\xi})$ for $\xi \in \psi$.
  There exists some set $G \in \binom{F}{n}$ such that $\xi \in \psi_{Gh} = \psi \cap \bigl[ (Gh)^{\partial n} \bigr]$, and we have $\varphi_{Fh,\psi,\xi} = \varphi_{Gh,\psi,\xi} \prod (F\setminus G)$ and $\deg \varphi_{Gh,\psi,\xi} = D + nd - n$.
  The coefficients of $\varphi_{Gh,\psi,\xi}$ are algebraic over $dX$ and the coefficients of $Gh$, with finitely many possibilities for the irreducible polynomials (since all possibilities for $G$ are isomorphic).
  The valuations of the coefficients of these irreducible polynomials are $O\bigl( \alpha - \tilde{v}(h) \bigr)$ in absolute value, by weak $v$-genericity.
  By Newton's polygon we deduce that $\tilde{v}(\varphi_{Gh,\psi,\xi}) = O\bigl( \alpha - \tilde{v}(h) \bigr)$ as well.
  For any $\zeta \in \bP^n(L)$, we have, by \autoref{lem:VirtualDivisorDefinitionInequality}:
  \begin{gather}
    \label{eq:VirtualDivisorBasis}
    v\varphi_{Fh,\psi,\xi}(\zeta) \geq (m-n)d \eta(\zeta) + v\varphi_{Gh,\psi,\xi}(\zeta).
  \end{gather}
  Therefore,
  \begin{align*}
    \eta^*(\varphi_{Fh,\psi,\xi})
    = {} & \inf_{\zeta \in \bP^n(L)} \bigl( v\varphi_{Fh,\psi,\xi}(\zeta) - (D + md-n) \eta(\zeta) \bigr) \\
    \geq {} & \inf_{\zeta \in \bP^n(L)} \bigl( v\varphi_{Gh,\psi,\xi}(\zeta) - (D + nd-n) \eta(\zeta) \bigr) \\
    \geq {} & \tilde{v}(\varphi_{Gh,\psi,\xi}) = O\bigl( \alpha - \tilde{v}(h) \bigr).
  \end{align*}
  For the converse inequality, observe that equality holds in \autoref{eq:VirtualDivisorBasis} for $\zeta = \xi$ (since $\xi \ind_K F \setminus G$).
  If $g = \sum_{\xi \in \psi} a_\xi \varphi_{Fg,\psi,\xi}$, then:
  \begin{align*}
    \eta^*(g)
    \leq {} & \min \bigl\{ vg(\xi) - (D+md-n)\eta(\xi) : \xi \in \psi \bigr\} \\
    = {} & \min \bigl\{ v(a_\xi) + v\varphi_{Fh,\psi,\xi}(\xi) - (D+md-n)\eta(\xi) : \xi \in \psi \bigr\} \\
    = {} & \min \left\{ v(a_\xi) + v\varphi_{Gh,\psi,\xi}(\xi) - (D+nd-n)\eta(\xi) : G \in \binom{F}{n}, \ \xi \in \psi_{Gh} \right\} \\
    = {} & \min_{\xi \in \psi} v(a_\xi) + O\bigl( \alpha - \tilde{v}(h) \bigr).
  \end{align*}
  Together, these two inequalities conclude the proof.
\end{proof}

\begin{lem}
  \label{lem:VirtualDivisorDualExplicit}
  Continuing \autoref{lem:VirtualDivisorSetup}, let $h \in K[X]_D$.
  Then for any $G \in \binom{F}{n}$, and for some constant $C = C(n,d)$, we have
  \begin{gather}
    \label{eq:VirtualDivisorDualExplicitFormula}
    \eta^*(h) = \min \bigl\{ vh(\xi) - D \eta(\xi) : \xi \in [G^{\partial n}] \bigr\},
    \\
    \label{eq:VirtualDivisorDualExplicitBound}
    vh(\xi) \leq \tilde{v}(h) + C D w_d(\eta) \qquad \forall \xi \in [F^{\partial n}].
  \end{gather}
\end{lem}
\begin{proof}
  Let $h_k = h^{kd} X_0^{n(d-1)}$, of degree $m_k d - n$, where $m_k = kD + n$.
  In what follows, for $O(\cdot)$ notation, we consider all the data fixed, with the exception of $k$, which may vary.
  On the one hand, since $\eta^*$ is a sub-valuation, we have
  \begin{gather*}
    \eta^*(h_k) \geq n(d-1) \eta^*(X_0) + dk \eta^*(h) = \frac{m_k d}{D} \eta^*(h) + O(1).
  \end{gather*}
  On the other hand, in $\eta^*(h) = \inf_\xi vh(\xi) - D \eta(\xi)$, the infimum is obtained along some sequence of $\xi_\ell$.
  For each $\ell$ there is some $i \leq n$ such that $\widehat{X}_i(\xi_\ell) = 0$ (i.e., $\tilde{v}(x_\ell) = v(x_{\ell,i})$).
  Possibly passing to a sub-sequence and permuting coordinates, we may assume that $\widehat{X}_0(\xi_\ell) = 0$ throughout.
  But then,
  \begin{gather*}
    \eta^*(h_k) \leq \min_\ell \ \bigl[ vh_k(\xi_\ell) - (m_k d - n)\eta(\xi_\ell) \bigr] = \frac{m_k d}{D} \eta^*(h) + O(1).
  \end{gather*}
  It follows that
  \begin{gather*}
    \lim_{k \rightarrow \infty} \frac{D \eta^*(h_k)}{m_k d} = \eta^*(h).
  \end{gather*}

  Fixing $k$ for a while, let $F = F_k = (f_i : i < m_k)$, let $\psi \subseteq [F^{\partial n}]$ be good, and let $h_k = \sum_{\xi \in \psi} a_\xi \varphi_{F,\psi,\xi}$.
  For each $\xi = [x] \in [F^{\partial n}]$ there is some $G_\xi \in \binom{F}{n}$ such that $\xi \in [G_\xi^{\partial n}]$.
  Let us identify $F$ and $G_\xi$ with their products, so $F = (F/G_\xi) G_\xi$.
  Since $\partial_j G_\xi(x)$ vanish for $j < n$, we have $\partial_n F(x) = (F/G_\xi)(x) \partial_n G_\xi(x)$, and therefore $v \partial_n F(\xi) = kDd \eta(\xi) + v \partial_n G_\xi(\xi)$.
  Therefore, letting $k$ vary and keeping everything else fixed, we have
  \begin{align*}
    \eta^*(h_k)
    = {} & \min_{\xi \in \psi} v(a_\xi) + O(1) \\
    = {} & \min_{\xi \in \psi} \bigl( vh_k(\xi) - v\partial_n F_k (\xi) \bigr) + O(1) \\
    = {} & \min_{\xi \in \psi} \bigl( vh_k(\xi) - kDd \eta(\xi) - v\partial_n G_\xi (\xi) \bigr) + O(1).
  \end{align*}
  The constant in $O(1)$ does not depend on the choice of good set $\psi$, so we may take the minimum over $\xi \in [F^{\partial n}]$.
  But then, by symmetry over $K$, we have for $G = \{f_0 : i < n\}$ (of any other $n$-family of the $f_i$):
  \begin{align*}
    \eta^*(h_k)
    = {} & \min_{\xi \in [G^{\partial n}]} \bigl( vh_k(\xi) - kDd \eta(\xi) - v\partial_n G (\xi) \bigr) + O(1) \\
    = {} & \min_{\xi \in [G^{\partial n}]} \bigl( m_kd vh(\xi) / D - m_k d \eta(\xi) \bigr) + O(1) \\
    = {} & \frac{m_k d}{D} \min_{\xi \in [G^{\partial n}]} \bigl( vh(\xi) - D \eta(\xi) \bigr) + O(1).
  \end{align*}
  Therefore,
  \begin{gather*}
    \lim_{k \rightarrow \infty} \frac{D \eta^*(h_k)}{m_k d} = \min_{\xi \in [G^{\partial n}]} \bigl( vh(\xi) - D \eta(\xi) \bigr).
  \end{gather*}
  This proves \autoref{eq:VirtualDivisorDualExplicitFormula}.

  For \autoref{eq:VirtualDivisorDualExplicitBound}, let $\xi \in [F^{\partial n}]$, so $\xi \in [G^{\partial n}]$ for some $G \in \binom{F}{n}$, and let $g = \prod G$.
  Let $N = n! \binom{nd}{n}$, which is the cardinal of the multi-set $[G^{\partial n}]$, and let $C = \sum_{i < n} \frac{d n N}{nd-i}$.
  We have $\tilde{v}(g) = n \tilde{v}(f_0) = -d n w_d(\eta)$.
  For $i < n$, the degree of $\partial_i g$ in $G^{\partial n}$ is $N/(nd-i)$, so the total degree of $g$ in $G^{\partial n}$ is $C/dn$.
  It follows that $\tilde{v}(G^{\partial n}) \geq -C w_d(\eta)$.

  Let $A = K(dX, G)$ (i.e., the $K$-algebra generated by $dX$ and the coefficients of $G$), and let $B = K[dX, \lambda_{i,j} : i<n, j < d]$, where the $\lambda_{i,j}$ are indeterminate linear forms (again, adjoining $\lambda_{i,j}$ means adjoining its coefficients).
  We equip $B$ with the valuation $v_B = \tilde{v}$, i.e., we make the (coefficients of) $\lambda_{i,j}$ $v$-generic over $K$.
  Finally, we define a map $s\colon A \rightarrow B$ fixing $K[dX]$ and sending $f_i \mapsto \prod_j \lambda_{i,j}$ (we assume, as we may, that $G = (f_i : i < n)$).
  Since each $f_i$ is weakly $v$-generic over everything else, and $\tilde{v}(\lambda_{i,j}) \geq 0$, we have $v_A \leq v_B \circ s$.
  In addition, $s(G^{\partial n})$ is then $n!$-power of the product of formal intersections of $n$ among the $\lambda_{i,j}$, i.e., a product of $N$ many $v$-generic algebraic points (times some scalar which is a power of $dX$ and can be ignored).
  Therefore,
  \begin{gather*}
    v(h \wedge G^{\partial n})
    \leq v_B\left( h \wedge \left( \prod_{i,j} \lambda_{i,j} \right)^{\partial n} \right)
    = N \tilde{v}(h).
  \end{gather*}
  Since $vh(\zeta) \geq \tilde{v}(h)$,
  \begin{align*}
    vg(\xi)
    \leq {} & \sum_{\zeta \in [G]^{\partial n}} vg(\zeta) - (N-1) \tilde{v}(h)
    \\
    = {} & v(h \wedge G^{\partial n}) - D \tilde{v}(G^{\partial n}) - (N-1) \tilde{v}(h)
    \\
    \leq {} & N \tilde{v}(h) + C D w_d(\eta) - (N-1) \tilde{v}(h)
    \\
    = {} & \tilde{v}(h) + C D w_d(\eta),
  \end{align*}
  as claimed.
\end{proof}

We can now start proving our main result, in several steps.
The first step is presented mostly for expository purposes, in order to present the main ingredients in a simpler setting.

\begin{lem}
  \label{lem:MainResultPn}
  \autoref{thm:MainResult} holds when $W = \bP^n$.

  Moreover, there is no need to adjoin the indeterminates $T$ or assume that the valuation is good.
\end{lem}
\begin{proof}
  We may assume that $\fC_{\bP^n}$ is the one with the canonical normalisation (otherwise just multiply some vector in each basis by an appropriate scalar).
  We will show that \autoref{eq:MainResult} holds for degrees of the form $md-n$, for $m \geq 2n$, where $B_{md-n} = \fM_{md-n}$ is the set of monomials.
  Throughout, let $m \geq n$ and $N = \dim K[X]_{md-n} = \binom{md}{n}$.

  By \autoref{lem:MainResultObservations}, we may assume that $\sup \eta = 0$, and re-scaling, we may assume that $w_d(\eta) = 1$.
  Let $L$, $dX$, and $F = (f_i : i < m)$ be as per \autoref{lem:VirtualDivisorSetup}.
  In particular, $F$ is good, so let $\psi \subseteq [F^{\partial n}]$ be a good set.
  Identifying $L[X]_{md-n}$ with $L^N$ via the basis $\Phi_{F,\psi}$, we have $\eta^* = \tilde{v} + O(1)$, and so $\vol_{\Phi_{F,\psi}} \eta^* = O(N) = O(m^n)$.

  By \cite[Theorem~4.15]{BenYaacov:Vandermonde}:
  \begin{gather*}
    \det \Phi_{F,\psi} = \fd_{F,\psi} = \prod_{H \in \binom{F}{n+1}} H^\wedge \prod_{G \subseteq F, \ |G| \leq n} \widehat{\fd}_{G,\psi}.
  \end{gather*}
  Since $m \geq 2n$, we have $\binom{m}{k} \leq \binom{m}{n}$ for $k \leq n$, so we have fewer than $(n+1) \binom{m}{n} = O(m^n)$ factors of the form $\widehat{\fd}_{G,\psi}$.
  Each $\widehat{\fd}_{G,\psi}$ is algebraic over $dX$ and the coefficients of $G$, with only finitely many possibilities for its irreducible polynomial over them.
  Since all the $f_i$ are $O(1)$-generic over $K$, we have $\bigl| v(\widehat{\fd}_{G,\psi}) \bigr| = O(1)$ (by Newton's polygon).
  Therefore
  \begin{gather*}
    v(\det \Phi_{F,\psi})
    = \sum_{H \in \binom{F}{n+1}} v(H^\wedge) + O(m^n)
    = d^{n+1} \binom{m}{n+1} \eta^{\wedge n+1} + O(m^n).
  \end{gather*}
  We have $\frac{(md-n)N}{n+1} = \binom{md}{n+1} = d^{n+1} \binom{m}{n+1} + O(m^n)$, and $\eta^{\wedge n+1} = O(1)$ (since $-1 \leq \eta \leq 0$), so:
  \begin{gather*}
    d^{n+1} \binom{m}{n+1} \eta^{\wedge n+1} = \frac{(md-n)N}{n} \eta^{\wedge n+1} + O(m^n).
  \end{gather*}

  We conclude that
  \begin{align*}
    \vol_{\fM_{md-n}} \eta^*
    = {} & \vol_{\Phi_{F,\psi}} \eta^* - v(\det \Phi_{F,\psi}) \\
    = {} & - d^{n+1} \binom{m}{n+1} \eta^{\wedge n+1} + O(m^n) \\
    = {} & - \frac{(md-n)N}{n+1} \eta^{\wedge n+1} + O(m^n).
  \end{align*}
  Dividing by $(md-n)N$ we obtain the desired estimate.
\end{proof}

In the second step, we consider a hypersurface.
The argument follows a similar path, with several added technical complications.

\begin{lem}
  \label{lem:MainResultHypersurface}
  \autoref{thm:MainResult} holds when $W \subseteq \bP^n$ is a hypersurface.

  Moreover, there is no need to adjoin the indeterminates $T$ or assume that the valuation is good.
\end{lem}
\begin{proof}
  We have $\fC_W = g \wedge \fC_{\bP^n}$, with $g \in K_0[X]_D$ being the irreducible polynomial defining $W$.
  Notice that then $\tilde{v}(g) = \tilde{v}(\fC_W)$, and we shall identify $\widehat{g}$ with $\widehat{\fC}_W$.

  For $m \geq D$ we have $K[W]_m = K[X]_m / g K[X]_{m-D}$, for which we construct a basis $B_m$ as follows.
  The set of monomials $\fM_m$ is a basis for $K[X]_m$, while $g \fM_{m-D}$ is a basis for $g K[X]_{m-D}$.
  We define $B^0_m \subseteq K[X]_m$ to be any set which completes $g \fM_{m-D}$ to a basis for $K[X]_m$, such that the transition matrix to $\fM_m$ has determinant $\pm 1$.
  Its image $B_m$ in $K[W]_m$ is a basis there.
  We also let the tuple $a$ be $(1,g)$, so $\tilde{v}(a) = 0 \wedge \tilde{v}(g)$.

  By \autoref{lem:MainResultObservations}, we may assume that $\sup \eta = 0$, and re-scaling, we may assume that $w_d(\eta) = 1$.
  It is then easy to construct $\theta$ of degree $d$ on $\bP^n$ such that $\eta = \theta\rest_W$, $\sup \theta = 0$ and $w_d(\theta) = -\inf \theta = w_d(\eta)$.
  Then $\theta^*$ is a sub-valuation on $K[X]$, and on $K[W]_m$ we have $\theta^*_m / g K[X]_{m-D} \leq \eta^*_m$ (since $\theta^*_m \leq \eta^*_m$ by definition, and $g K[X]_{m-D} \subseteq \ker \eta^*_m$).
  By \autoref{lem:Volume}\autoref{item:VolumeQuotient}:
  \begin{gather*}
    \vol_{B_m} \eta^*_m \geq \vol_{B_m} \bigl( \theta^*_m / g K[X]_{m-D} \bigr) = \vol_{\fM_m} \theta^*_m - \vol_{g \fM_{m-D}} \theta^*_m.
  \end{gather*}
  On the other hand, $\eta^{\wedge n} \wedge \widehat{g} = \theta^{\wedge n} \wedge \widehat{g}$.

  We will show that \autoref{eq:MainResult} holds for degrees of the form $md+D-n$, for $m \geq n$.
  Throughout, let $M = \dim K[W]_{md+D-n} = \binom{md+D}{n} - \binom{md}{n}$.
  Then we want to show that:
  \begin{gather}
    \label{eq:MainResultHypersurface}
    \frac{\vol_{\fM_{md+D-n}} \theta^* - \vol_{g \fM_{md-n}} \theta^*}{(md+D-n) M} + \frac{\theta^{\wedge n} \wedge \widehat{g}}{n} \geq \frac{\gamma}{m}\bigl( v(a) - w_d(\eta) \bigr).
  \end{gather}
  Let $L$, $dX$, and $F = (f_i : i < m)$ correspond to $\theta$ as per \autoref{lem:VirtualDivisorSetup}.

  Let us see how \autoref{eq:MainResultHypersurface} changes when we multiply $g$ by a constant $b \in K$.
  We have
  \begin{gather*}
    \frac{\theta^{\wedge n} \wedge \widehat{bg}}{n}
    = \frac{\theta^{\wedge n} \wedge \widehat{g} + v(b)/D}{n}
    = \frac{\theta^{\wedge n} \wedge \widehat{g}}{n} + \frac{v(b)}{nD},
    \\
    \vol_{b g \fM_{md-n}} \theta^*
    = \vol_{g \fM_{md-n}} \theta^* + \binom{md}{n} v(b).
  \end{gather*}
  Observe that
  \begin{align*}
    (md+D-n)M
    = {} & (n+1) \binom{md+D}{n+1} - (n+1)\binom{md}{n+1} - D\binom{md}{n}
    \\
    = {} & nD \binom{md}{n} + (n+1)\left[ \binom{md+D}{n+1} - \binom{md}{n+1} - D\binom{md}{n} \right]
    \\
    = {} & nD \binom{md}{n} + (n+1) \sum_{i<D} \ \left[ \binom{md+i}{n} - \binom{md}{n} \right],
  \end{align*}
  so for some constants $c,c'$:
  \begin{gather}
    \label{eq:MainResultHypersurfaceM}
    (md + D - n)M - nD \binom{md}{n} \in [0, cm^{n-1}],
    \qquad
    \frac{(md + D - n)M - nD \binom{md}{n}}{nD (md+D-n)M} \in [0, c'/m].
  \end{gather}
  Now choose $b$ such that $v(b) = -\tilde{v}(g)-CD$, where $C = C(n,d)$ is as per \autoref{lem:VirtualDivisorDualExplicit}.
  If $v(b) \leq 0$, then the left hand side of \autoref{eq:MainResultHypersurface} decreases, and $\tilde{v}(g) \geq -CD$, so $\tilde{v}(a)$ in the right hand side changes by at most $CD$.
  And if $v(b) \geq 0$, then $\tilde{v}(g) \leq - CD$, and now the left hand side increases by at most $v(b) O(m^{-1})$, and $\tilde{v}(a)$ in the right hand side increases by $v(b)$.
  In either case, choosing $\gamma$ large enough, we may reduce to the case where $\tilde{v}(g) = -CD$.

  For the purpose of calculating the left hand side, we may adjoin to $K$ new $v$-generic elements, so let $g_0 \in K[X]_D$ be $v$-generic over $K_0$, and let $h = g + g_0$.
  Then, by \autoref{lem:VirtualDivisorDualExplicit}, we have $vg(\xi) = v h (\xi)$ for each $\xi \in [F^{\partial n}]$, so $v(g\chi)(\xi) = v (h \chi) (\xi)$ for every polynomial $\chi \in K[X]_{md-n}$ and every $\xi \in [F^{\partial n}]$, and therefore $\theta^*(g\chi) = \theta^*(h \chi)$.
  We conclude that $\vol_{g \fM_{md-n}} \theta^* = \vol_{h \fM_{md-n}} \theta^*$, and since $[f_0 \wedge \ldots \wedge f_{n-1}] \subseteq [F^{\partial n}]$, also $\theta^{\wedge n} \wedge \widehat{g} = \theta^{\wedge n} \wedge \hat{h}$.
  We may therefore replace $g$ with $h$, and assume that $g$ is also weakly $v$-generic (over the prime field).

  We have thus reduced to the hypotheses \autoref{lem:VirtualDivisorBasis}.
  In particular, $Fg$ is a good family, and let us fix a good set $\psi \subseteq \bigl[ (Fg)^{\partial n} \bigr]$.
  Then
  \begin{gather*}
    \frac{\det \Phi_{Fg,\psi}}{\det \Phi_{F,\psi}} = \prod_{H \in \binom{F}{n}} g \wedge H^\wedge \prod_{G \subseteq F, \ |G| \leq n-1} \widehat{\fd}_{Gg,\psi}.
  \end{gather*}
  As in the proof of \autoref{lem:MainResultPn} we have $v(\widehat{\fd}_{Gg,\psi}) = O(1)$, so
  \begin{align*}
    v(\det \Phi_{Fg,\psi} / \det \Phi_{F,\psi_F})
    & {} = \sum_{H \in \binom{F}{n}} v(g \wedge H^\wedge) + O(m^{n-1}) \\
    & {} = D d^n \binom{m}{n} \theta^{\wedge n} \wedge \widehat{g} + O(m^{n-1}) \\
    & {} = D \binom{md}{n} \theta^{\wedge n} \wedge \widehat{g} + O(m^{n-1}) \\
    & {} = (md+D-n) M \frac{\theta^{\wedge n} \wedge \widehat{g}}{n} + O(m^{n-1}).
  \end{align*}

  Notice that $g\Phi_{F,\psi} \subseteq \Phi_{Fg,\psi}$, so let $\Phi' = \Phi_{Fg,\psi} \setminus g\Phi_{F,\psi} = \bigl( \varphi_{Fg,\psi,\xi} : \xi \in \psi \setminus \psi_F \bigr)$, generating $E \subseteq L[X]_{md+D-n}$.
  By \autoref{lem:VirtualDivisorBasis}, the three valuations $\theta^*_{md+D-n}$, $\theta^*_{md+D-n}/gL[X]_{md-n}$ and $\tilde{v}$ (with respect to the basis $\Phi'$) agree on $E$ up to $O(1)$.
  It follows that
  \begin{gather*}
    \vol_{\Phi_{Fg,\psi}} \theta^* - \vol_{g\Phi_{F,\psi}} \theta^*
    =
    \vol_{\Phi'} \bigl( \theta^* / g L[X]_{md-n} \bigr)
    = O(M) = O(m^{n-1}).
  \end{gather*}

  We conclude that
  \begin{align*}
    \vol_{\fM_{md+D-n}} \theta^* - \vol_{g\fM_{md-n}} \theta^*
    = {} & \vol_{\Phi_{Fg,\psi}} \theta^* - \vol_{g\Phi_{F,\psi}} \theta^* - v(\det \Phi_{Fg,\psi}/\det \Phi_{F,\psi}) \\
    = {} & - (md+D-n) M \frac{\theta^{\wedge n} \wedge \widehat{g}}{n} + O(m^{n-1}).
  \end{align*}
  Dividing by $(md+D-n)M$, we obtain \autoref{eq:MainResultHypersurface}.
\end{proof}

\section{Generic projections of varieties}
\label{sec:GenericProjection}

We shall prove \autoref{thm:MainResult} by reducing the general case to that of a hypersurface, via a generic projection.
A projective variety of dimension $\nu-1$ can always be projected onto a hypersurface in $\bP^\nu$, so let us study the behaviour of virtual divisors under such a projection.

Throughout, let $\nu \leq n$, $X = (X_0,\ldots,X_n)$ and $Y = (Y_0,\ldots,Y_\nu)$.
As in \autoref{thm:MainResult}, we let $K_0$ be a small algebraically closed field, and $K = K_0(T)^a$ for some infinite family of indeterminates $T$.
Let $W \subseteq \bP^n(K)$ be a projective variety of dimension $\nu-1$ defined over $K_0$, with Chow form $\fC_W$.

For $i \leq \nu$, let $\mu_i = \sum_{j \leq m} T_{i,j} X_j$, where $T_{i,j}$ are distinct members of the family $T$.
We are going to work in a large ambient algebraically closed field $L$ containing $K$.
For any polynomial $f \in L[Y]$ define $f_\mu(X) = f(\mu X) \in L[X]$.
Define
\begin{gather}
  \label{eq:VarietyToHypersurfacePolynomial1}
  P(Y)
  = \fC_W\left( \frac{\mu_i}{Y_i} - \frac{\mu_\nu}{Y_\nu} : i < \nu \right) \prod_{i\leq \nu} Y_i^{\deg W}.
\end{gather}
Recall that by \cite[Lemma~2.8]{BenYaacov:Vandermonde}, this is a polynomial in $K[Y]_{\deg W}$, invariant, up to sign, under permutations of $Y$.
In addition, by \cite[Fact~2.10(v)]{BenYaacov:Vandermonde}, for a point $\xi = [x] \in \bP^n$ algebraically independent over $K_0$ from $K$, we have $P_\mu(x) = 0$ if and only if $\xi \in W$.

\begin{lem}
  \label{lem:VarietyToHypersurface}
  \begin{enumerate}
  \item
    \label{item:VarietyToHypersurfaceDefined}
    If $\xi = [x] \in W$, then $\mu x \neq 0$, so $\pi \xi = [\mu x]$ defines a map $\pi\colon W \rightarrow \bP^\nu$.
  \item
    \label{item:VarietyToHypersurfaceInjective}
    This map is generically injective, namely if $\xi \in W$ is generic over $K$, then the fibre of $\xi$ is a singleton.
  \item
    \label{item:VarietyToHypersurfaceImage}
    Let $U = \pi W \subseteq \bP^\nu$ denote the image.
    The polynomial $P$ defined above is irreducible, and $V(P) = U \subseteq \bP^\nu$.
    In particular, $\deg U = \deg W$ as well.
  \item
    \label{item:VarietyToHypersurfaceFunctions}
    There exists a homogeneous polynomial $R \in K[Y]$ such that, for every homogeneous $f \in L[X]$ there exists $g \in L[Y]$ such that $g_\mu - f R_\mu \in I(W)$.
    In particular, $\pi$ is injective outside $V(R_\mu)$.
    Moreover, there exists a tuple $a$ in $K$ such that, if $L$ is equipped with a good valuation, then
    \begin{gather*}
      \tilde{v}(g) \geq \tilde{v}(R) + \tilde{v}(f) + \tilde{v}(a)\deg f.
    \end{gather*}
  \item
    \label{item:VarietyToHypersurfaceWedge}
    For any homogeneous polynomials $f = (f_i : i < \nu)$ in $L[Z]$ we have
    \begin{gather*}
      f \wedge P = f_\mu \wedge \fC_W.
    \end{gather*}
  \end{enumerate}
\end{lem}
\begin{proof}
  The intersection of $\nu + 1$ (or even just $\nu$) generic hyperplanes with $W$ is empty, so $\xi = [x] \in W$ implies $\mu x \neq 0$, whence \autoref{item:VarietyToHypersurfaceDefined}.
  Given $\upsilon = [y] \in \bP^n$ such that $y_\nu \neq 0$, its fibre is the common zero-set of the forms
  \begin{gather*}
    \lambda_i^\upsilon = \mu_i - \frac{y_i}{y_\nu} \mu_\nu, \qquad i < \nu.
  \end{gather*}
  Therefore $P(y) = 0$ if and only if $\upsilon \in U$.
  Since $P$ is invariant (up to sign) under permutations, we may always assume that $y_\nu \neq 0$, so $U = V(P)$.

  Consider now a point $\xi \in W$ generic over $K$, or equivalently, generic over $K_0$ and algebraically independent from $K$.
  The linear forms $\lambda^{\pi \xi}$ are then generic over $K_0(\xi)$, or even $K_0(\xi,\mu_\nu)$, modulo the constraint that they vanish at $\xi$, so $W \cap V(\lambda^{\pi \xi}) = \{\xi\}$, and the fibre is a singleton, proving \autoref{item:VarietyToHypersurfaceInjective}.

  Since $W$ is irreducible, so is $U$, so $P = P_0^r$ for some irreducible $P_0 \in K[Y]$.
  Intersecting $U$ with $\nu-1$ generic hyperplanes we obtain $\deg P_0$ distinct generic points.
  Pulling back to $\bP^n$, we obtain the intersection of $W$ with $\nu-1$ generic hyperplanes, consisting of $\deg W$ distinct generic points.
  Since the fibre of a generic point of $W$ is a singleton, we must have $\deg U = \deg P_0 = \deg W$, i.e., $r = 1$ and $P$ is irreducible, proving \autoref{item:VarietyToHypersurfaceImage}.

  For \autoref{item:VarietyToHypersurfaceFunctions}, let us consider the variety $\widehat{W} \subseteq \bP^{\nu+n+1}$ consisting of all $[x,\mu x]$ such that $[x] \in W$.
  The map $\pi$ factors naturally via $W \rightarrow \widehat{W} \rightarrow U$.
  Let $\xi = [x] \in W$ be again generic over $K$.
  The linear forms $(\lambda_i^{\pi \xi} : i < \nu-1)$ are generic over $K_0$.
  Considering them as linear forms in $X,Y$, the intersection $\lambda_0^{\pi \xi} \wedge \ldots \wedge \lambda_{\nu-2}^{\pi \xi} \wedge \fC_{\widehat{W}}$ splits into distinct points, on of which is $\hat{\xi} = [x,\mu x]$.
  By \cite[Fact~2.10(iv)]{BenYaacov:Vandermonde}, $K(\hat{\xi})$ is separable over $K_0(\lambda_i^{\pi \xi} : i < \nu-1)$, and \textit{a fortiori} over $K(\pi \xi)$.
  On the other hand, since $\pi$ is generically injective, $K(\hat{\xi})$ is purely inseparable over $K(\pi \xi)$.
  We conclude that $\hat{\xi}$ is rational over $K(\pi \xi)$.
  We can therefore express each $x_j / \mu_\nu x$ as a zero-degree rational function over $K$ in $\mu x$.

  We also claim that for each $j \leq n$ there exists a homogeneous polynomial $X_j^d + \ldots \in K[X_j,Y]$ satisfied by $(x_j,\mu x)$.
  Indeed, if not, then $[1,0]$ specialises $[x_j,\mu x]$ over $K$, so there exists $x' \neq 0$ such that $[x',0]$ specialises $[x,\mu x]$ over $K$.
  In particular, $[x'] \in W$ and $\mu x' = 0$, contradicting \autoref{item:VarietyToHypersurfaceDefined}.
  This proves our claim, which implies that for every $f(X) \in K[X]_d$ there exists $g \in K[X,Y]_d$ such that $f(X) - g(X, \mu X) \in I(W)$ and $\deg_X g \leq d_0$, where $d_0$ only depends on $W$.
  Together with the previous paragraph, we obtain \autoref{item:VarietyToHypersurfaceFunctions}.
  The moreover part follows from the argument.

  For \autoref{item:VarietyToHypersurfaceWedge}, if $f^\wedge = 0$, i.e., $\dim V(f) > 0$, then $f \wedge P = 0 = f_\mu \wedge \fC_W$.
  We may therefore assume that $f^\wedge \neq 0$, in which case it splits into finitely many points.
  Then $f \wedge P = 0$ if and only if one of these points is in the image of $W$, i.e., if and only if the fibre over one of these points intersects $W$.
  But $V(f_\mu)$ is the union of these fibres, so $f \wedge P = 0$ if and only if $f_\mu \wedge \fC_W = 0$.
  Thinking of $f$ as indeterminate polynomials, both $f \wedge P$ and $f_\mu \wedge \fC_W$ are polynomials in these indeterminates.
  They are irreducible by \cite[Proposition~2.15]{BenYaacov:Vandermonde}, and have the same zeros.
  Therefore, they only differ by a scalar coefficient.
  To find this coefficient, consider first the case where $f_i = Y_i$, so $f_{i,\mu} = \mu_i$.
  Then
  \begin{gather*}
    f \wedge P
    = (-1)^{\nu D} P \wedge Z_0 \wedge \ldots \wedge Z_{\nu-1}
    = P( 0,\ldots,0, 1 )
    = \fC_W( \mu_i : i < \nu  )
    = f_\mu \wedge \fC_W.
  \end{gather*}
  We obtain the general case, for arbitrary degrees, by specialising $f_i$ to $Y_i^{\deg f_i}$.
\end{proof}

\begin{dfn}
  \label{dfn:VarietyToHypersurfacePull}
  Let $L$ be equipped with a valuation, and let $\theta$ be a virtual divisor on $U(K)$, and define
  \begin{gather*}
    \mu_0 = \min_i \hat{\mu}_i\rest_{W(K)},
    \qquad
    \pi^* \theta = \theta \circ \pi + \mu_0.
  \end{gather*}
  We call $\pi^* \theta$ the \emph{pull-back} of $\theta$ to $W$.
\end{dfn}

Clearly, $\mu_0$ is a virtual divisor (or degree one) on $W$, also defined by $\mu_0(\xi) = \tilde{v}(\mu x) - \tilde{v}(x)$.
If $\theta = \hat{f}\rest_{U(K)}$ where $f \in L[X]_d$, and $\xi = [x] \in W(K)$, then
\begin{gather*}
  \hat{f}_\mu(\xi)
  = \frac{v\bigl( f(\mu x)\bigr)}{d} - \tilde{v}(x)
  = \hat{f}(\pi \xi) + \mu_0(\xi)
  = \theta(\pi \xi) + \mu_0(\xi).
\end{gather*}
Therefore $\pi^* \theta$ is also a virtual divisor on $W$, of the same degree as $\theta$.

\begin{lem}
  \label{lem:VarietyToHypersurfacePull}
  Assume that $L$ is equipped with a valuation, and let $\theta$ be a virtual divisor on $U$.
  \begin{enumerate}
  \item
    \label{item:VarietyToHypersurfacePullVolume}
    We have
    \begin{gather*}
      \theta^{\wedge \nu} \wedge \widehat{P} = (\pi^* \theta)^{\wedge \nu} \wedge \widehat{\fC}_W.
    \end{gather*}
  \item
    \label{item:VarietyToHypersurfacePullSubValuation}
    If $u = \theta^*$ and $w = (\pi^* \theta)^*$, then $u(g) = w(g_\mu)$ for homogeneous $g \in K[U]$.
    % If $u$ is (almost) finitely generated by a family of conditions $u(g_k) \geq \gamma_k$, then $w$ is (almost) finitely generated by the conditions $w(g_{k,\mu}) \geq \gamma_k$.
  \end{enumerate}
\end{lem}
\begin{proof}
  Item \autoref{item:VarietyToHypersurfacePullVolume} is just \autoref{lem:VarietyToHypersurface}\autoref{item:VarietyToHypersurfaceWedge}.
  For \autoref{item:VarietyToHypersurfacePullSubValuation}, let $g \in L[U]_d$.
  Then
  \begin{align*}
    u(g)
    = d \inf_U \ \bigl( \hat{g} - \theta \bigr)
    = d \inf_W \ \bigl( \hat{g}_\mu - \mu_0 - \theta \circ \pi \bigr)
    = d \inf_W \ \bigl( \hat{g}_\mu - \pi^* \theta \bigr)
    = w(g_\mu).
  \end{align*}
  % For the last assertion we may assume that all the $\gamma_k$ vanish, so $\theta = \min_k \hat{g}_k\rest_U$.
  % Then $\pi^* \theta = \min_k \hat{g}_{k,\mu}\rest_W$, so $w$ is generated by the conditions $w(g_{k,\mu}) \geq 0$.
\end{proof}

\begin{lem}
  \label{lem:VarietyToHypersurfacePullGood}
  If the valuation on $L$ is good, then $\mu_0 = 0$.
  In particular, $\pi^* \theta = \theta \circ \pi$ for any virtual divisor $\theta$ on $U$.
\end{lem}
\begin{proof}
  It will suffice to prove that if $\xi = [x] \in W$ and $\tilde{v}(x) = 0$, then $\tilde{v}(\mu x) = 0$, in each of the two cases described in \autoref{dfn:VarietyToHypersurfaceGood}.
  We let $\overline{\xi} = [\overline{x}] \in \bP^n(\overline{K})$, so in the residue field we have $\trdeg_{\overline{K_0}} \overline{K_0}(\overline{\xi}) \leq \nu-1$.
  The linear forms $\mu_i$ are over the valuation ring, and descend to the residue field.
  If there is any $I \subseteq \nu+1$ of size $\nu$ such that $(\mu_i : i \in I)$ are $v$-generic over $K_0$, then the residues $(\overline{\mu}_i : i \in I)$ are generic over $\overline{K_0}$, and so $\overline{\mu x} \neq 0$ as in the argument for \autoref{lem:VarietyToHypersurface}\autoref{item:VarietyToHypersurfaceDefined}.

  In case \autoref{item:VarietyToHypersurfaceGoodGeneric}, the entire family $\mu$ is $v$-generic.
  In case \autoref{item:VarietyToHypersurfaceGoodPAdic}, we assume that $K_0(T)$ is equipped with a $P$-adic valuation for some prime $P \in K_0[T]$.
  If some of the coefficients of some $\mu_i$, say $\mu_\nu$ appear in the polynomial $P$, then $\mu_{<\nu}$ are $v$-generic over $K_0$ (in fact, the valuation on $K_0(\mu_{<\nu})$ is trivial).
  Otherwise, again the entire family $\mu$ is $v$-generic over $K_0$, concluding the proof.
\end{proof}

Push virtual divisors from $W$ to $U$ is more delicate.
Indeed, if $\eta$ is a virtual divisor on $W$, one could imagine letting $u = \eta^*$ and defining a sub-valuation $\pi_* u(f) = u(f_\mu)$ on $K[U]$.
However, it is not at all clear why $\pi_* u$ should be almost finitely generated.
Similarly, there may be a (non-generic) fibre which is not a singleton, and on which $\eta$ need not be constant, so $\eta$ need not be a limit of $\pi^* \theta$.
We therefore follow a somewhat longer route.

\begin{lem}
  \label{lem:VarietyToHypersurfacePush}
  There exists a non-zero tuples $a$ in $K$, which depends only on $W$, with the following property.
  For every $d$ there exists $D$, which depends only on $d$ and $W$, such that for any good valuation on $K$ and any virtual divisor $\eta$ of degree $d$ on $W$, defined over $K_0$, there exists a virtual divisor $\theta$ of degree $D$ on $U$, defined over $K$, satisfying:
  \begin{gather*}
    \sup \theta = \sup \eta, \qquad \pi^*\theta \geq \eta, \qquad w_D(\theta) \leq w_d(\eta) - \tilde{v}(a),
    \\
    \theta^{\wedge \nu} \wedge \widehat{P} \leq \eta^{\wedge \nu} \wedge \widehat{\fC}_W + \frac{w_d(\eta) - \tilde{v}(a)}{d}.
  \end{gather*}
\end{lem}
\begin{proof}
  Let $a$ and $R \in K[Y]_e$ be as per \autoref{lem:VarietyToHypersurface}\autoref{item:VarietyToHypersurfaceFunctions}, and we may assume that $\tilde{v}(R) = 0$.
  Let $f \in L[X]_d$ such that $\eta = \hat{f}\rest_{W(K_0)}$, chosen such that $K_0(f)$ is independent from $K$ over $K_0$, in the sense of \autoref{sec:Definability}.
  Then the extension of $\eta$ to $W(K)$ is $\hat{f}\rest_{W(K)}$.
  We may assume that $\sup \eta = 0$ and $\tilde{v}(f) = - d w_d(\eta)$.
  Let $m \in \bN$ be fixed later, and let $D = md + e$.
  Let $g \in L[Y]_D$ be as per \autoref{lem:VarietyToHypersurface}\autoref{item:VarietyToHypersurfaceFunctions}, so $g_\mu = f^m R_\mu$ on $W$ and $\tilde{v}(g) \geq m\tilde{v}(f) + md \tilde{v}(a) = md \tilde{v}(a) - md w_d(\eta)$.
  Let $h_0 \in L[Y]_D$ be $v$-generic over $K(g)$, let $h = g + h_0$, and let $\theta = \hat{g}\rest_{U(K)} \wedge 0 = \hat{h}\rest_{U(K)}$.
  Then $\theta$ is a virtual divisor on $U$ of degree $D$, and
  \begin{gather*}
    w_D(\theta) \leq - \tilde{v}(h)/D \leq w_d(\eta) - \tilde{v}(a).
  \end{gather*}
  For $\xi \in W$ we have
  \begin{gather*}
    \pi^* \theta(\xi)
    = 0 \wedge \frac{md \eta(\xi) + e \hat{R}(\pi \xi)}{D}
    \geq \frac{md}{D} \eta(\xi)
    \geq \eta(\xi).
  \end{gather*}

  For the last inequality, let $k < \nu$.
  Then
  \begin{align*}
    \widehat{R}_\mu \wedge \eta^{\wedge k} \wedge (\pi^* \theta)^{\wedge \nu-k-1} \wedge \widehat{\fC}_W
    & {} \leq \frac{\tilde{v}(R_\mu \wedge \fC_W)}{e \deg W},
    \\
    \eta^{\wedge k+1} \wedge (\pi^* \theta)^{\wedge \nu-k-1}\wedge \widehat{\fC}_W
    & {} \geq -(k+1) w_d(\eta) - (\nu-k-1) w_D(\theta) \geq \nu \tilde{v}(a) - \nu w_d(\eta).
  \end{align*}
  Since $\pi^* \theta = \hat{h}_\mu\rest_{W(K)} \leq \hat{g}_\mu\rest_{W(K)}$:
  \begin{align*}
    \eta^{\wedge k} \wedge (\pi^* \theta)^{\wedge \nu-k} \wedge \widehat{\fC}_W
    & {} \leq \eta^{\wedge k} \wedge (\pi^* \theta)^{\wedge \nu-k-1} \wedge \hat{g}_\mu \wedge \widehat{\fC}_W
    \\
    & {} = \frac{md}{D} \eta^{\wedge k + 1} \wedge (\pi^* \theta)^{\wedge \nu-k-1} \wedge \widehat{\fC}_W
      + \frac{e}{D} \hat{R}_\mu \wedge \eta^{\wedge k} \wedge (\pi^* \theta)^{\wedge \nu-k-1} \wedge \widehat{\fC}_W
    \\
    & {} \leq \eta^{\wedge k + 1} \wedge (\pi^* \theta)^{\wedge \nu-k-1} \wedge \widehat{\fC}_W
      + \frac{e}{D} \left( \frac{\tilde{v}(R_\mu \wedge \fC_W)}{e \deg W} + \nu w_d(\eta) - \nu \tilde{v}(a) \right).
  \end{align*}
  Therefore
  \begin{gather*}
    \theta^{\wedge \nu} \wedge \widehat{P}
    = (\pi^* \theta)^{\wedge \nu} \wedge \widehat{\fC}_W
    \leq \eta^{\wedge \nu} \wedge \widehat{\fC}_W
    + \frac{\nu}{D \deg W} \tilde{v}(R_\mu \wedge \fC_W)
    + \frac{\nu^2 e}{D}\bigl( w_d(\eta) - \tilde{v}(a) \bigr).
  \end{gather*}
  Extending $a$ with the inverse of some non-zero coefficient appearing in $R_\mu \wedge \fC_W$, and choosing $m$ large enough, the last inequality follows.
\end{proof}

\section{Proof of the main result}
\label{sec:MainResultProof}

We can now prove our main result, \autoref{thm:MainResult}.
We are given $d \in \bN$ and $W \subseteq \bP^n$ of dimension $\nu-1$ defined over $K_0$, and we work in $K = K_0(T)^a$.
We seek a tuple $a$ in $K$, and for each $\gamma > 0$ some large enough $m$ and a basis $B_m$ for $K[W]_m$, such that for any good valuation on $K$ (i.e., such that $T$ is $v$-generic over $K_0$) and any $\eta$ of degree $d$ defined over $K_0$, we have
\begin{gather*}
  \frac{\vol_{B_m} \eta^*}{m \dim K[W]_m} \geq - \frac{\eta^{\wedge \nu} \wedge \widehat{\fC}_W}{\nu} + \gamma \bigl( \tilde{v}(a) - w_d(\eta) \bigr).
\end{gather*}
If $W = \bP^n$, this is \autoref{lem:MainResultPn} (we may also consider $\bP^n$ as a hyperplane in $\bP^{n+1}$, so \autoref{lem:MainResultPn} is not strictly necessary).
Otherwise, we let $\nu = \dim W + 1 \leq n$.
Let $U \subseteq \bP^\nu(K)$ be the generic projection of $W$, defined over $K$ by a polynomial $P$, as in \autoref{sec:GenericProjection}.
Let $D$ be as per \autoref{lem:VarietyToHypersurfacePush}, and let $a'$ denote the corresponding tuple in $K$ (called $a$ there).

By \autoref{lem:MainResultHypersurface}, there is a tuple $a''$ in $K$, and for arbitrarily large $m$ we have bases $C_m$ for $K[U]_m$ such that for any valuation on $K$ and virtual divisor $\theta$ on $U$ of degree $D$ defined over $K_1$, we have
\begin{gather*}
  \frac{\vol_{C_m} \theta^*}{m \dim K[U]_m} \geq - \frac{\theta^{\wedge \nu} \wedge \widehat{P}}{\nu} + \bigl( \tilde{v}(a'') - w_D(\theta) \bigr) O(m^{-1}).
\end{gather*}
We define $B_m$ to consist of $(C_m)_\mu = \{ h_\mu : h \in C_m \}$, completed with some monomials to a basis of $K[W]_m$, and let $a$ consist of all products of members of $a'$ and $a''$ (we may assume that $1$ appears in each one).

Consider now a good valuation on $K$ and $\eta$ of degree $d$ on $W$, and as usual we may assume that $\sup \eta = 0$.
We choose $\theta$ on $U$ of degree $D$ as per \autoref{lem:VarietyToHypersurfacePush}.
Notice that $\eta \leq 0$ implies that $\tilde{v} \leq \eta^*$, and $\pi^* \theta \geq \eta$ implies $(\pi^* \theta)^* \leq \eta^*$.
Therefore
\begin{gather*}
  \vol_{C_m} \theta^*
  = \vol_{(C_m)_\mu} (\pi^*\theta)^*
  \leq \vol_{(C_m)_\mu} \eta^*
  \leq \vol_{B_m} \eta^*.
\end{gather*}
We therefore have, for some constant $\delta$:
\begin{align*}
  \frac{\vol_{B_m} \eta^*}{m \dim K[U]_m}
  \geq {} & \frac{\vol_{C_m} \theta^*}{m \dim K[U]_m}
  \\
  \geq {} & - \frac{\theta^{\wedge \nu} \wedge \widehat{P}}{\nu} + \frac{\delta}{m} \bigl( \tilde{v}(a'') - w_D(\theta) \bigr)
  \\
  \geq {} & - \frac{\eta^{\wedge \nu} \wedge \widehat{\fC}_W}{\nu} + \frac{\tilde{v}(a') - w_d(\eta)}{d\nu} + \frac{\delta}{m} \bigl( \tilde{v}(a'') + \tilde{v}(a') - w_d(\eta) \bigr)
  \\
  \geq {} & - \frac{\eta^{\wedge \nu} \wedge \widehat{\fC}_W}{\nu} + \left( \frac{1}{d \nu} + \frac{\delta}{m} \right) \bigl( \tilde{v}(a) - w_d(\eta) \bigr).
\end{align*}
We may always replace $d$ with a multiple, and therefore assume it is large enough, and then choosing $m$ large enough we have:
\begin{align*}
  \frac{\vol_{B_m} \eta^*}{m \dim K[U]_m}
  \geq {} & - \frac{\eta^{\wedge \nu} \wedge \widehat{\fC}_W}{\nu} + \frac{\gamma}{2} \bigl( \tilde{v}(a) - w_d(\eta) \bigr).
\end{align*}
Finally, observe that $\dim K[U]_m \leq \dim K[W]_m$ are both $\deg W \binom{m}{\nu-1} + O(m^{\nu-2})$, and since $\eta \geq -w_d(\eta)$, we have $- \frac{\eta^{\wedge \nu} \wedge \widehat{\fC}_W}{\nu} \leq w_d(\eta)$.
Therefore, for large enough $m$:
\begin{align*}
  \frac{\vol_{B_m} \eta^*}{m \dim K[W]_m}
  \geq {} & - \frac{\eta^{\wedge \nu} \wedge \widehat{\fC}_W}{\nu} + \gamma \bigl( \tilde{v}(a) - w_d(\eta) \bigr),
\end{align*}
concluding the proof.

% \begin{lem}
%   Let $\eta$ be a virtual divisor on $W$ and $\sC$ a virtual chain of dimension zero in $W$.
%   Then for every $f \in K[W]_m$:
%   \begin{gather*}
%     \frac{\eta^*(f)}{m} \leq \hat{f} \wedge \sC - \eta \wedge \sC.
%   \end{gather*}
% \end{lem}

% \begin{rmk}[XXX For later XXXX]
%   Let $(f_i : i < n) \in L^1(\mu)$ be such that $\int f_i d\mu = 0$.

%   \begin{gather*}
%     \bigvee_{i < n} f_i + f_n
%     = \bigvee_{i \leq n} f_i  + \bigvee_{i < n} (f_i \wedge f_n)
%   \end{gather*}
% \end{rmk}

% \begin{thm}
%   Let $g \in K[X]_d$ be such that 
%   \begin{gather*}
%     \frac{\eta^*\bigl( \fM_m^\wedge \bigr) - \eta^*\bigl( (g\fM_{m-d})^\wedge \bigr)}{\binom{m+n}{n} - \binom{m+n-d}{n}}
%     + m \frac{\eta^{\wedge n} \wedge g^v}{\ell \deg g} + c(g) + ....? \geq 0
%   \end{gather*}
%   % \begin{gather*}
%   %   \frac{\vol\bigl( K[X]_m, \eta^*, \fM_m \bigr) - \vol\bigl( gK[X]_{m-d}, \eta^*, g \fM_{m-d} \bigr)}{m\bigl( \dim K[X]_m - \dim K[X]_{m-d}\bigr)}
%   %   + \frac{\eta^{\wedge n} \wedge g^v}{nd} \geq -\frac{c}{m}.
%   % \end{gather*}
% \end{thm}

\bibliographystyle{begnac}
\bibliography{begnac}

\end{document}